\documentclass{amsart}
\usepackage{setspace}
\usepackage[margin=1in]{geometry}
\usepackage{epsfig}
\usepackage{lscape}  
\usepackage[justification=raggedright]{caption}	
\usepackage{amsmath,amsthm,amssymb,amsfonts,amsopn,mathrsfs,mathtools}
\usepackage{tikz}
\usetikzlibrary{arrows}
\usepackage[all]{xy}
\usepackage{hyperref}
\usepackage{color}
\usepackage{tabulary}
\usepackage{cite}
\usepackage[normalem]{ulem}
\usepackage{aliascnt}
\usepackage{todonotes}
\usepackage{kpfonts}

\makeatletter

\newcommand{\Z}{\mathbb Z}

\newcommand{\mbb}{\mathbb}
\newcommand{\mc}{\mathcal}
\newcommand{\dual}{\,\check{}}
\newcommand{\RGr}{\mathbb R\mathrm{Gr}}

\newcommand{\Set}{\mathbf{Set}}
\newcommand{\Sm}[1]{\mathbf{Sm}_{#1}}
\newcommand{\Sch}[1]{\mathbf{Sch}_{#1}}

\newcommand{\St}{\mathrm{St}}

\newcommand{\ptr}{\mathrm{ptr}}

\newcommand{\Sp}{\mathbf{Sp}}
\newcommand{\mothom}{\mathbf{H}}
\newcommand{\KR}{\mathbf{KR}^{\mathrm{alg}}}
\newcommand{\BigKR}{\uline{\mathbf{KR}^{\mathrm{alg}}}}
\newcommand{\Quad}{\mathbf{Quad}}

\DeclareMathOperator{\Spec}{Spec}

\DeclareMathOperator*{\colim}{colim}
\DeclareMathOperator*{\im}{im}

\DeclareMathOperator{\Hom}{Hom}

\DeclareMathOperator{\rk}{rk}

\DeclareMathOperator{\can}{can}
\DeclareMathOperator{\iHom}{\mathbf{Hom}}
\DeclareMathOperator{\Ch}{Ch}
\DeclareMathOperator{\Fun}{Fun}
\DeclareMathOperator{\Vect}{Vect}
\DeclareMathOperator{\Cone}{Cone}
\DeclareMathOperator{\sPerf}{sPerf}
\DeclareMathOperator{\KGL}{KGL}
\DeclareMathOperator{\SH}{SH}
\DeclareMathOperator{\HH}{H}
\DeclareMathOperator{\BigSH}{\uline{SH}}
\DeclareMathOperator{\BigHH}{\uline{H}}

\theoremstyle{plain}

\newtheorem{lemma}{Lemma}
\numberwithin{lemma}{section}
\newtheorem{theorem}[lemma]{Theorem}

\theoremstyle{definition}
\newtheorem{corollary}[lemma]{Corollary}
\newtheorem{example}[lemma]{Example}
\newtheorem{definition}[lemma]{Definition}
\newtheorem{notation}[lemma]{Notation}

\newtheorem{remark}[lemma]{Remark}

\newtheorem{porism}[lemma]{Porism}

\title{Cdh Descent for Homotopy Hermitian $K$-Theory of Rings with Involution}
\author{Daniel Carmody}

\begin{document}

\maketitle

\begin{abstract}
We provide a geometric model for the classifying space of 
automorphism groups of Hermitian vector bundles over a ring with
involution $R$ such that $\frac{1}{2} \in R$; this generalizes a
result of Schlichting-Tripathi \cite{SchTri}. We then prove a
periodicity theorem for
Hermitian $K$-theory and use it to construct an $E_\infty$ motivic ring spectrum $\KR$
representing homotopy Hermitian $K$-theory. From these results, we
show that $\KR$ is stable under base change, and cdh descent for
homotopy Hermitian $K$-theory of rings with involution is a formal consequence.
\end{abstract}

\tableofcontents
\section{Introduction}

Algebraic $K$-theory is an algebraic invariant introduced in the 1950s by Alexander
Grothendieck where it served as the cornerstone of his reformulation
of the Riemann-Roch theorem \cite{Gro57}.  Twenty years previously, Ernst Witt
developed the notion of quadratic forms over arbitrary fields and
introduced the Witt ring as an object to encapsulate the nature of
all the quadratic forms over a given field \cite{Witt1937}. Combining
the ideas of Grothendieck and Witt, Hyman Bass introduced a category of quadratic
forms  $\Quad(R)$
with isometries over a ring $R$ and studied $K_1(\Quad(R))$ and $K_0(\Quad(R))$. $K_0(\Quad(R))$ is what we know today as the
Grothendieck-Witt ring, and Bass was able to recover the Witt ring
as a quotient of $K_0(\Quad(R))$ by the image of the hyperbolic
quadratic forms. He went on to show that $K_1(\Quad(R))$ was related to the stable
structure of the automorphisms of hyperbolic modules, which complemented the relationship between $K_1(R)$ and the group $GL(R)$. The $K$-theory
of quadratic forms soon found applications to surgery theory where
the periodic $L$-groups defined by Wall in 1966 \cite{Wall66} served
as obstructions to certain maps being cobordant to homotopy equivalences. When
the means to define the higher algebraic $K$-groups via the $+$ construction was discovered by
Quillen in the 1970s, Karoubi applied it to the orthogonal groups $BO$
in order to define the higher Hermitian $K$-theory of rings with
involution as we know it today\cite{Kar73}.

Fast forward twenty years into the 1990s when Morel and Voevodsky
developed the motivic homotopy category and proved that algebraic
$K$-theory was representable in the stable motivic homotopy category
\cite{MorelVoev}. The development of the stable motivic homotopy
category not only gave a new domain to motivic cohomology, it also
opened the door for applications of topological tools like obstruction
theory to
more algebraic objects. Several
subsequent developments inspire our work here.

The first set of developments relates to Hermitian $K$-theory. In 2005 Hornbostel showed that Hermitian $K$-theory was
representable in the stable motivic homotopy category on schemes
\cite{Hornbos2005}. We note that Hornbostel defined
Hermitian $K$-theory on schemes by extending the definition on rings
using Jouanolou's trick. In 2011 Hu-Kriz-Ormsby showed that Hermitian
$K$-theory on the category of $C_2$-schemes over a field is representable in the
$C_2$-equivariant stable motivic homotopy category \cite{HuKriz}. Here
they used a similar trick to Hornbostel in order to extend Hermitian
$K$-theory from rings with involution to schemes with involution. In
the meantime, Schlichting, building off of work of Thomason,
Karoubi, and Balmer, defined the higher Hermitian $K$-theory of a dg-category
with weak equivalences and duality and proved the analogues of the
fundamental theorems of higher $K$-theory for these groups
\cite{Schder}. Although some of Schlichting's theorems are stated only
for schemes (rather than schemes with $C_2$ action), many of his
proofs require only trivial modification to extend to Grothendieck-Witt
groups of schemes with $C_2$ action. See also \cite{Xie2018ATM} for
the proofs of the equivariant version of some of the theorems together with a new
transfer morphism. Another approach is taken by Hesselholt-Madsen, who
define real algebraic $K$-theory of a category with weak equivalences
and duality as a symmetric spectrum object in the monoidal category of
pointed $C_2$-spaces. Schlichting's higher Grothendieck-Witt groups can be
recovered from the Hesselholt-Madsen construction by taking
homotopy groups of $C_2$-fixed points of deloopings of the real
algebraic $K$-theory spaces with respect to the sign representation spheres. We note 
as well that the Ph.D. thesis of Alejo López-Ávila \cite{LopezHermInfty} shows that 
the motivic spectrum representing Hermitian $K$-theory in the non-equivariant setting 
has an $E_\infty$ structure.

Back in $K$-theory land, Cisinski proved that the six
functor formalism in motivic homotopy theory developed by Ayoub
\cite{Ayoub} together with the fact that the motivic $K$-theory
spectrum $\KGL$ is a cocartesian section of $\SH(-)$ yields a simple
proof of cdh-descent (descent in the completely decomposed h topology) for homotopy $K$-theory \cite{DenRep}. This in
turn yields a short proof of Weibel's vanishing conjecture for homotopy
$K$-theory, and inspired work of Kerz, Strunk, and Tamme who solved
Weibel's conjecture by proving pro-cdh descent for
ordinary $K$-theory \cite{KerzStrunkTamme}. Hoyois in \cite{cdhdesc}
uses Cisinski's approach to show cdh descent for equivariant homotopy $K$-theory.

This paper, inspired by the above developments, shows cdh-descent
for homotopy Hermitian $K$-theory of schemes with $C_2$ action. The
techniques in \cite{cdhdesc} provide our pathway to descent. In order
to show that Hermitian $K$-theory is a cocartesian section of
$\SH^{C_2}(-)$, we need to show that the Hermitian $K$-theory space
$\Omega^\infty GW$ can be represented by a certain Grassmannian, and
we need a periodization theorem in order to pass from the Hermitian
$K$-theory space $\Omega^\infty
GW$ to the homotopy Hermitian $K$-theory motivic spectrum $L_{\mbb
  A^1}\mbb GW$. Schlichting and Tripathi \cite{SchTri} show that
$\Omega^\infty GW$ is representable by a Grassmannian over schemes
with trivial action over a regular base scheme with 2 invertible. Their techniques extend to the equivariant
setting, and with slight modification provide a proof of
representability over non-regular bases. The periodization techniques
in \cite{cdhdesc} extend to Hermitian $K$-theory by
investigating the Hermitian $K$-theory of $T^\rho$, the Thom
space of the regular representation $\mbb A^\rho$.

\subsection{Outline}

Section \ref{chap:Eq_Top} begins with a review of $G$-equivariant motivic
homotopy theory where $G$ is a finite group scheme over a base $S$
which is Noetherian of finite Krull dimension, has an ample
family of line bundles, and has $\frac{1}{2} \in \Gamma(S,\mc O_S)$. First we review the definition of the equivariant
\'etale and Nisnevich topologies, then we introduce the isovariant
\'etale topology and give some examples of covers. For the reader
familiar with non-equivariant motivic homotopy theory, the assumptions
we make on $G$ are strong enough so that structural results are mostly
the same:
\begin{itemize}
\item the equivariant Nisnevich topology is generated by a nice cd-structure,
\item equivariant schemes are locally affine in the equivariant
  Nisnevich topology, and
\item to invert $G$-affine bundles $Y \rightarrow S$ it suffices to
  invert $\mbb A_S^1$.
\end{itemize}

The content in this section is a selection of relevant content from
\cite{GrpSchHell}. We end this section with the definition of the unstable and stable
equivariant motivic $\infty$-categories \ref{sec:cat_def} a la Hoyois
\cite{HoyoisSixOp}.

Section \ref{chap:herm_form_sch} reviews the definitions and results
on Hermitian forms which will be necessary to work with the
Grothendieck-Witt spectrum. Section \ref{sec:herm_def} contains the
basic definitions and examples, while section
\ref{sec:herm_form_semiloc} contains the tools necessary to show that
Hermitian forms are locally determined by rank in the isovariant or equivariant
\'etale topologies. The final section \ref{sec:Higher_GW} reviews the
main definitions of \cite{Schder} to allow us to talk about the
Grothendieck-Witt spectra of schemes with involution. 

Section \ref{chap:herm_grass} is where the background material ends
and the paper begins in earnest. We combine the techniques of \cite{SchTri}
and \cite{cdhdesc} in order to show that classifying spaces of automorphism
groups of Hermitian vectors bundles are representable in the $C_2$-equivariant motivic
homotopy category.

This section
culminates with the representability result, Theorem \ref{thm:RGr_GW_nonreg}, which we note holds
over non-regular base schemes:

\begin{theorem}
Let $S$ be a Noetherian scheme of finite Krull dimension with an ample
family of line bundles and $\frac{1}{2} \in S$. There is an
equivalence of motivic spaces on $\Sm{S,qp}^{C_2}$
\[
L_{\mathrm{mot}}\Z \times \RGr_\bullet \xrightarrow{\sim}
L_{\mathrm{mot}}\Z \times \colim_n B_{isoEt}O(\mbb H^n).
\]
\end{theorem}

With a simple modification to remove the regularity hypothesis, one can follow
\cite{SchTri} to show that
\[
L_{\mathrm{mot}}\Z \times \colim_n B_{isoEt}O(\mbb H^n)
\xrightarrow{\sim} L_{\mathrm{mot}} \Omega^\infty GW
\]
but as this is unnecessary for proving cdh descent, we leave it out of this paper.

Section \ref{chap:Einf} provides a convenient way of passing from the
presheaf of Grothendieck-Witt spectra to an
$E_\infty$-motivic spectrum in $\SH^{C_2}(S)$. The crucial fact is that
the localizing version of Hermitian $K$-theory of rings with involution, denoted $\mbb
GW$, is the periodization of $GW$ with respect to a certain Bott map
derived from projective bundle formulas for $\mbb P^1$ and $\mbb P^\sigma$ (see Corollary \ref{cor:GW_per}). Here
$\mbb P^\sigma$ is a copy of $\mbb P^1$ with action $[x:y] \mapsto
[y:x]$. The fact that the periodization functor is monoidal together
with Schlichting's results on monoidality of $GW$ immediately give
that the motivic spectrum $L_{\mbb A^1}\mbb GW \in \SH^{C_2}(S)$ is an
$E_\infty$ object \ref{thm:GW_Einfty}.

\begin{theorem}
Let $S$ be a Noetherian scheme of finite Krull dimension with an ample
family of line bundles and $\frac{1}{2} \in S$. Then $L_{\mbb A^1}\mbb GW$ lifts to an $E_\infty$ motivic spectrum, denoted
$\KR$, over $\Sm{S,qp}^{C_2}$.
\end{theorem}

The final section \ref{chap:cdh} follows the recipe given by Cisinski
and summarized in \cite{cdhdesc} to prove cdh descent for equivariant
homotopy Hermitian $K$-theory on the category of quasi-projective $S$-schemes. After reviewing the $K$-theory case,
the section culminates in theorem \ref{thm:cdh_desc}.

\begin{theorem}
Let $S$ be a Noetherian scheme of finite Krull dimension with an ample
family of line bundles and $\frac{1}{2} \in S$. Then the
homotopy Hermitian $K$-theory spectrum of rings with involution $L_{\mbb A^1}\mbb
GW$ satisfies descent for the equivariant cdh topology on $\Sch{S,qp}^{C_2}$.
\end{theorem}
{\bf Acknowledgements} 

This paper is a condensed and cleaned up version of my doctoral thesis. 
I owe a debt of gratitude to my advisor Jeremiah Heller for his smooth and continuous support 
throughout my time in graduate school. I would also like to thank my academic siblings Tsutomu Okano and Brian Shin
for the many helpful conversations about motivic homotopy theory. Finally, I want to thank Elden Elmanto
for his help and willingness to answer even my most inane questions.

This paper is much better off because of the detailed reading and helpful comments of an anonymous reviewer, and 
I would like to thank them for their work.

\section{Equivariant Topologies and the Equivariant Motivic Homotopy Category}\label{chap:Eq_Top}

This section reviews the foundations of equivariant motivic homotopy
theory. The key definitions are those of the equivariant \'etale and
Nisnevich topologies -- two topologies that play a crucial role in
defining the equivariant motivic infinity category
$\HH^{G}(S)$ over a Noetherian base scheme $S$ with finite Krull
dimension, with an ample family of line bundles, and with $\frac{1}{2}
\in \Gamma(S,\mc O_S)$. Throughout we'll work with two categories of schemes.
Let $\Sch{S,qp}^{C_2}$ denote the category of quasi-projective $C_2$-schemes which are 
separated and finite type over $S$, and let 
 $\Sm{S,qp}^{C_2}$ be the full subcategory of schemes smooth over $S$.

\begin{notation}
Throughout this section, $G$ will be either a finite group or the
group scheme over $S$ associated to a finite group. Recall that to pass between
finite groups and group schemes over $S$, we form the scheme
$\coprod_G S$ with multiplication (using that fiber products commute
with coproducts in $\Sch{S,qp}$):
\[
\xymatrix{\coprod_G S \times_S \coprod_G S \ar[r]^-\sim &\coprod_{(g_1,g_2)\in G \times
  G} S \ar[r]^-{\mu}& \coprod_G S}
\]

Whenever we write down a pullback square involving schemes, we'll
tacitly be thinking of $G$ as a group scheme, and $X \times Y$ will
really mean $X \times_S Y$.
\end{notation}

We introduce the background definitions from \cite{GrpSchHell} which
will allow us to define the isovariant \'etale topology. This is a topology
which is slightly coarser than the equivariant \'etale topology,
but whose points are still nice enough so that Hermitian vector
bundles are locally determined by rank. 

\begin{definition}
For a $G$-scheme $X$, the isotropy group scheme is a group scheme
$G_X$ over $X$ defined by the cartesian square
\[
\xymatrix{G_X \ar[r] \ar[d] & G \times X \ar[d]^{(\mu_X,id_X)} \\ X
  \ar[r]^{\Delta_X} & X \times X}
\]
\end{definition}

\begin{definition}
Let $X$ be a $G$-scheme. The scheme-theoretic stabilizer of a point
$x$ in $X$ is the pullback
\[
\xymatrix{G_x \ar[r] \ar[d] & G_X \ar[d]\\ \Spec k(x) \ar[r] & X.}
\]

By the pasting lemma, this is the same as the pullback
\[
\xymatrix{G_x \ar[r] \ar[d] & G \times X \ar[d]\\ \Spec k(x) \ar[r] &
  X \times X}
\]
\end{definition}

\begin{definition}
Let $X$ be a $G$-scheme, and define the \emph{set-theoretic}
stabilizer $S_x$ of $x \in X$ to be $\{g \in G
\mid gx = x\}$. 
\end{definition}

\begin{remark}
With notation as above, the underlying set of the scheme-theoretic
stabilizer $G_x$ can be described as
\[
G_x = \{g \in S_x \mid \text{ the induced morphism $g : k(x)
  \rightarrow k(x)$ equals } id_{k(x)}\}.
\]
\end{remark}

The example below shows that set-theoretic and scheme-theoretic
stabilizers need not agree.

\begin{example} (Herrmann \cite{GrpSchHerr})
Let $k$ be a field, and consider the $k$-scheme given by a finite
Galois extension $k \hookrightarrow L$. Let $G = Gal(L/k)$ be the
Galois group. The set-theoretic stabilizer of the unique point in
$\Spec L$ is $G$ itself, while the scheme-theoretic stabilzer is
$\{e\} \subset G$. 
\end{example}

\begin{remark}
Recall that if $Z \rightarrow X$ is a monomorphism of schemes, then
the forgetful functor from schemes to sets preserves any pullback $Z
\times_X Y$. The forgetful functor $\Sch{S,qp}^G \rightarrow \Sch{S,qp}$
is a right adjoint, hence preserves pullbacks.

Since the inclusion of a point $\Spec k(x) \hookrightarrow X \times_S X$
will be a monomorphism, the difference
between the set-theoretic and scheme-theoretic stabilizers is due to
the fact that the underlying space of $X \times_S X$ is not
necessarily the fiber product of the underlying spaces. Indeed, in the
example above, $\Spec L \times_k \Spec L \cong \coprod_{g \in G} \Spec
k$, whereas the pullback in spaces is just a single point. 
\end{remark}

\subsection{The Equivariant and Isovariant \'Etale Topologies}

\begin{notation}
Let $S$ be a  $G$-scheme. The equivariant \'etale topology on $\Sch{S,qp}^{G}$
will denote the site whose covers are \'etale covers whose component
morphisms are equivariant.
\end{notation}

\begin{definition} (Thomason)
An equivariant map $f : Y \rightarrow X$ is said to be
\emph{isovariant} if it induces an isomorphism on isotropy groups $G_Y \cong G_X \times_X
Y$. A collection $\{f_i : X_i \rightarrow X\}_{i \in I}$ of
equivariant maps called an isovariant \'etale cover if it is an
equivariant \'etale cover such that each $f_i$ is isovariant. 
\end{definition}

\begin{remark}
The isovariant topology is equivalent to the topology whose covers are
equivariant, stabilizer preserving, \'etale maps. We'll use this
notion more often in computations. 
\end{remark}

\begin{remark}
The points in the isovariant \'etale topology are schemes of the form
$G \times^{G_x} \Spec(\mc O^{sh}_{X,\overline x})$ where $\overline x
\rightarrow x \rightarrow X$ is a geometric point, and $(-)^{sh}$ denotes
strict henselization. See \cite{GrpSchHell} for a proof.
\end{remark}

The fact that the points in the isovariant \'etale topology are either
strictly henselian local or hyperbolic rings will be crucial when we want to describe
the isovariant \'etale sheafification of the category of Hermitian
vector bundles. Fortunately Hermitian vector bundles over such rings
are well understood, and we'll in fact show that Hermitian vector bundles
are up to isometry determined by rank locally in the isovariant
\'etale topology. 

\begin{remark}
If $G = C_2$, then $G_x = \{e\}$ or $G_x = C_2$ for all $x \in X$. If
$G_x = \{e\}$, then $G \times^{G_x} \Spec(\mc O^{sh}_{X,\overline x})
\cong C_2 \times \Spec(\mc O^{sh}_{X,\overline x}) \cong \Spec(\mc
O^{sh}_{X,\overline x})\coprod \Spec(\mc O^{sh}_{X,\overline x})$ with a
free action. If $G_x = C_2$, then $G \times^{G_x} \Spec(\mc
O^{sh}_{X,\overline x}) = \Spec(\mc
O^{sh}_{X,\overline x})$ where the induced action on the residue field is trivial.
\end{remark}

\subsection{The Equivariant Nisnevich Topology}

Similarly to the non-equivariant case, the equivariant Nisnevich
topology is defined by a particularly nice cd-structure. While there are
a few different definitions of this topology in the literature which
can give non Quillen equivalent model structures, we use the
definition from \cite{GrpSchHell}.

\begin{definition}
A distinguished equivariant Nisnevich square is a cartesian square in
$\Sch{S,qp}^G$
\[
\xymatrix{B \ar[r]\ar[d] & Y \ar[d]^p \\ A \ar@{^{(}->}[r]^i & X}
\]
where $i$ is an open immersion, $p$ is \'etale, and $p$ restricts to
an isomorphism
$(Y-B)_{\mathrm{red}} \rightarrow (X-A)_{\mathrm{red}}$. 
\end{definition}

\begin{definition}
The equivariant Nisnevich $cd$-structure on $\Sch{S,qp}^G$ is the
collection of distinguished equivariant Nisnevich squares in $\Sch{S,qp}^G$.
\end{definition}

The next remark has the important consequence that to prove a map is
an equivariant motivic equivalence, it suffices to check that it's an
equivalence on affine $G$-schemes. 

\begin{remark}
By Lemma 2.20 in \cite{GrpSchHell}, for finite groups $G$, any separated $G$-scheme of
finite type over $S$ is Nisnevich-locally
affine. 
\end{remark}

\subsection{The Equivariant cdh Topology}

The completely decomposed h (cdh) topology is, roughly speaking, the coarsest topology
satisfying Nisnevich excision and which allows for a theory of
cohomology with compact support. Like the Nisnevich topology (and
unlike the \'etale topology) it can be generated by a cd-structure,
which gives a convenient way to check whether or not a presheaf is a
cdh sheaf. 

\begin{definition}\label{def:cdh_top}
An abstract blow-up square is a cartesian square in $\Sch{S,qp}^G$
\[
\xymatrix{\widetilde Z \ar[r]\ar[d] & \widetilde X \ar[d]^p \\ Z \ar@{^{(}->}[r]^i & X}
\]
where $i$ is a closed immersion and $p$ is a proper map which induces
an isomorphism $(\widetilde X - \widetilde Z) \cong (X-Z)$.
\end{definition}

\begin{definition}
The cdh topology is the topology generated by the cd-structure whose
distinguished squares are
\begin{itemize}
\item the equivariant Nisnevich distinguished squares;
\item the abstract blowup squares.
\end{itemize}
\end{definition}

One canonical example of a cdh cover is the map $X_{red} \rightarrow
X$ for an equivariant scheme $X \rightarrow S$. Another example is given by resolution of singularities: given a
proper birational map $p : X \rightarrow Y$, it's an isomorphism over
some open set $U$ in $Y$, so letting $Z = Y-U$ and $\widetilde Z = X -
p^{-1}(U)$ we get an abstract blowup square
\[
\xymatrix{\widetilde Z \ar[r] \ar[d] & X \ar[d] \\ Z \ar[r] & Y}
\]

\subsection{The Equivariant Motivic Homotopy Category}\label{sec:cat_def}

In this paper, we'll work with a Noetherian scheme of finite Krull
dimension and a finite group scheme $G$ over $S$. Equivariant motivic
homotopy theory is developed in somewhat more generality by Hoyois in
\cite{HoyoisSixOp}, though there's a price to be paid for allowing
more general group schemes in that the motivic localization functor becomes more
complicated.  

\begin{definition}
A presheaf $F$ on $\Sch{S,qp}^G$ is called \emph{homotopy invariant} if
the projection $\mbb A^1_S \rightarrow S$ induces an equivalence $F(X)
\simeq F(X \times \mbb A^1)$ for each $X$ in $\Sch{S,qp}^G$. Denote by $\mc
P_{\mathrm{htp}}(\Sch{S,qp}^{G}) \subset \mc P(\Sch{S,qp}^G)$ the full
subcategory spanned by the homotopy invariant presheaves. Denote by
$L_{\mathrm{htp}}$ the corresponding localization endofunctor of $\mc P(\Sch{S,qp}^G)$.
When restricting to $\Sm{S,qp}^G$, we'll abuse notation and similarly let $L_{\mathrm{htp}}$
denote the corresponding localization endofunctor.
\end{definition}

Now we give the usual definition of excision, the condition that
guarantees that a presheaf is a Nisnevich sheaf.

\begin{definition}
A presheaf $F$ on $\Sch{S,qp}^G$ (or $\Sm{S,qp}^G$) is called \emph{Nisnevich excisive} if:
\begin{itemize}
\item $F(\emptyset)$ is contractible;
\item for every equivariant Nisnevich square $Q$ in $\Sch{S,qp}^G$ (or $\Sm{S,qp}^G$), $F(Q)$ is
  cartesian. 
\end{itemize}

Denote by $\mc P_{\mathrm{Nis}}(\Sm{S,qp}^G) \subset \mc P(\Sm{S,qp}^G)$ the
full subcategory of Nisnevich excisive presheaves. Denote by
$L_{\mathrm{Nis}}$ the corresponding localization endofunctor. 
\end{definition}

Finally we come to the definition of a motivic $G$-space, namely a
presheaf that is both Nisnevich excisive and homotopy invariant.

\begin{definition}
Let $S$ be a $G$-scheme. A \emph{motivic $G$-space} over $S$ is a
presheaf on $\Sm{S,qp}^G$ that is homotopy invariant and Nisnevich
excisive. Denote by $\HH^G(S) \subset \mc P(\Sm{S,qp}^G)$ the full
subcategory of motivic $G$-spaces over $S$.

Let 
\[
L_{{\mathrm{mot}}} = \colim_{n \rightarrow \infty} (L_{\mathrm{htp}}
\circ L_{\mathrm{Nis}})^n(F)
\]
denote the motivic localization functor, where the colimit is in the $\infty$-category
of presheaves.
\end{definition}

In order to form the stable equivariant motivic homotopy category, we
also need to discuss pointed motivic $G$-spaces.

\begin{definition}
Let $S$ be a $G$-scheme. A \emph{pointed motivic $G$-space} over $S$
is a motivic $G$-space $X$ over $S$ equipped with a global section $S
\rightarrow X$. Denote by $\HH_\bullet^G(S)$ the $\infty$-category of
pointed motivic $G$-spaces. 
\end{definition}

The definition of stabilization can in general be complicated. With
our assumptions however, we need only invert the Thom space of the
regular representation $T^\rho$.

\begin{definition}
Let $S$ be a $G$-scheme. The symmetric monoidal $\infty$-category of
\emph{motivic $G$-spectra} over $S$ is defined by 
\[
\SH^G(S) = \HH_\bullet^G(S)[( T^{\rho}_S)^{-1}] = \colim
\left(\HH_\bullet^G \xrightarrow{-\otimes T^\rho} \HH_\bullet^G
  \xrightarrow{-\otimes T^\rho} \cdots \right)
\]
where $T^\rho$ is the Thom space of the regular
representation $\mbb A^\rho/\mbb A^\rho - 0$ of $G$. The colimit is 
taken in the $\infty$-category of presentable $\infty$-categories.
\end{definition}

\subsection{Computations with Equivariant Spheres}

Because we'll be using equivariant spheres to index our spectra, we'll
record some of their basic properties here. These computations will be important when we investigate periodicity of $GW$ in section \ref{chap:Einf}. Though there are exotic
elements of the Picard group even in non-equivariant stable motivic
homotopy theory, we'll be concerned with the four building blocks
$S^1, S^\sigma = \colim (\ast \leftarrow (C_2)_+ \rightarrow S^0), \mbb G_m, \mbb
G_m^{\sigma}$. Here $\mbb G_m^\sigma$ is the $C_2$ scheme
corresponding to $S[T,T^{-1}]$ with action $T \mapsto T^{-1}$. 

\begin{lemma}
Let $\mbb P^\sigma$ denote $\mbb P^1$ with the action defined by $[x:y] \mapsto
[y:x]$. There is an equivariant Nisnevich square
\[
\xymatrix{C_2\times \mbb G_m^\sigma \ar[r]\ar[d]^{\pi_2} & \mbb P^1 -
  \{0\} \coprod \mbb P^1 - \{\infty\} \ar[d]^f \\ \mbb G_m^\sigma \ar[r]^i & \mbb P^\sigma}
\]
\end{lemma} 

\begin{proof}
Here, we identify $\mbb G_m^\sigma$ with $\mbb P^\sigma -
\{0,\infty\}$. The map $i$ is clearly an open immersion. Its
complement is $\mbb \{0,\infty\}$, and $f$ maps
$f^{-1}(\{0,\infty\})$ isomorphically onto
$\{0,\infty\}$. Furthermore, $f$ is a disjoint union of open
immersions, and hence is (in particular) \'etale. 
\end{proof}

\begin{lemma}
There is a homotopy pushout square, where $f$ can be taken to map $C_2$ to $\{[1:1]\}$:

\[
\xymatrix{(C_2)_+\wedge  (\mbb G_m^\sigma)_+ \ar[r]\ar[d]^{\pi_2} &
  (C_2)_+ \ar[d]^f \\ (\mbb G_m^\sigma)_+ \ar[r]^i & \mbb P_+^\sigma}
\]
\end{lemma}

\begin{proof}
The above square is equivalent to the square

\[
\xymatrix{(C_2)_+\wedge  (\mbb G_m^\sigma)_+ \ar[r]\ar[d]^{\pi_2} &
  (C_2)_+ \wedge \mbb A^1_+ \ar[d]^f \\ (\mbb G_m^\sigma)_+ \ar[r]^i & \mbb P_+^\sigma}.
\] 

By the lemma above, 

\[
\xymatrix{(C_2\times \mbb G_m^\sigma)_+ \ar[r]\ar[d]^{\pi_2} & (C_2
  \times \mbb A^1)_+\ar[d]^f \\ (G_m^\sigma)_+ \ar[r]^i & \mbb P_+^\sigma}
\]
is a homotopy pushout square. But adding a disjoint basepoint is a
monoidal functor, so $X_+ \wedge Y_+ \cong (X \times Y)_+$ and this
square is equivalent to the desired square.
\end{proof}

\begin{lemma}
$\mbb P^\sigma \approx S^\sigma \wedge \mbb G_m^\sigma$.
\end{lemma}

\begin{proof}
Let $Q$ denote the homotopy cofiber of $(C_2 \times \mbb G_m^\sigma)_+
\rightarrow (\mbb G_m^\sigma)_+$, and $\widetilde Q$ denote the homotopy
cofiber of $(C_2 \times \mbb A^1)_+ \rightarrow \mbb P^\sigma_+$. Then the
lemma above implies that $Q \approx \widetilde Q$. 

$Q$ is the homotopy cofiber of $
(C_2)_+ \wedge \mbb (G_m^\sigma)_+ \rightarrow S^0 \wedge \mbb
(G_m^\sigma)_+$, which is just $S^\sigma \wedge \mbb
(G_m^\sigma)_+$. Recall that $\colim(* \leftarrow X \rightarrow X
\wedge Y_+) \cong X \wedge Y$ since this is $X \wedge \colim(*
\leftarrow S^0 \rightarrow Y_+)$. Thus the cofiber of $S^\sigma
\rightarrow Q$ is $S^\sigma \wedge \mbb G_m^\sigma$.

The diagram below in which the horizontal rows are cofiber sequences
\[
\xymatrix{(C_2)_+ \ar[r]\ar[d]^{id} & S^0 \ar[d] \ar[r]& S^\sigma
  \ar[d] \\(C_2)_+ \ar[r]\ar[d] & \ar[r]\ar[d] \mbb P^\sigma_+ & \widetilde
  Q \ar[d] \\ \star \ar[r] & \mbb P^\sigma \ar[r] & T }
\]

implies that the cofiber of $S^\sigma \rightarrow \widetilde Q$ is
$\mbb P^\sigma$. 

The result now follows from the commutativity of the following diagram
and homotopy invariance of homotopy cofiber:
\[
\xymatrix{S^\sigma \ar[r]^{id}\ar[d] & S^\sigma \ar[d] \\ Q \ar[r]^{\sim} \ar[d]&
\widetilde Q \ar[d]\\ S^\sigma \wedge \mbb G_m^\sigma \ar[r] & \mbb P^\sigma}
\]

\end{proof}

\section{Hermitian Forms on Schemes}\label{chap:herm_form_sch}

This section reviews the definitions and properties of Hermitian forms over schemes
with involution from\cite{Xie2018ATM}. After defining the proper
notion of the dual of a quasi-coherent module over a scheme with involution, the definition of a Hermitian vector bundle finally
appears in Definition \ref{def:herm_form_sch} as a locally free $\mc
O_X$-module with a well-behaved map to the dual module. Once the
definitions are in place, we discuss in section
\ref{sec:herm_form_semiloc} the structure of Hermitian forms over
semilocal rings as this is the fundamental tool for showing that Hermitian forms
are locally trivial in the isovariant \'etale
topology. We prove this particular statement in Corollary
\ref{cor:loc_det_rank}. We end this section by recalling Schlichting's
definition of a dg categoy with weak equivalence and duality and the
Grothendieck-Witt groups of such an object.  

\subsection{Definitions}\label{sec:herm_def}

\begin{definition}
Let $R$ be a ring with involution $- : R \rightarrow R^{op}$. A \emph{Hermitian
module over $R$} is a finitely generated projective right $R$-module, $M$, together
with a map
\[
b : M \otimes_{\Z} M \rightarrow R
\]  

such that, for all $a \in R$,
\begin{enumerate}
\item $b(xa,y) = \overline a b(x,y)$,
\item $b(x,ya) = b(x,y) a$,
\item $b(x,y) = \overline{b(y,x)}$.
\end{enumerate}
\end{definition}

\begin{definition}
Let $R$ be a ring with involution $-$. Given a right $R$-module $M$, define
a left $R$-module, denoted $\overline M$ as follows: $\overline M$ has the same underlying abelian
group as $M$, and the action is given by $r \cdot m = m \cdot \overline r$. If $R$ is
commutative, we can promote $\overline M$ to an $R$-bimodule by introducing the right action
 $m \cdot r =
m \overline r$. 
\end{definition}

\begin{remark}
Let $R$ be a commutative ring so that $R = R^{op}$. Given an involution $\sigma : R \rightarrow R$ and a right $R$-module
$M$, we can identify $\overline M$ with $\sigma_*M$, where $\sigma_*M$ is 
the module $M$ with $R$-action restricted through $\sigma$. Typically the pushforward
would just take the right $R$-module $M$ to another right $R$-module. Since we really view 
$\sigma$ as landing in $R^{op}$, we use commutativity of $R$ and 
 the canonical identification of right $R^{op}$ modules 
with left $R$-modules to think of $\sigma_*M$ as a left module.
Indeed,
$\sigma_*M$ is a left $R$-module via the rule $r \cdot m =
m \cdot \sigma(r) $. 
\end{remark}

\begin{remark}
Another way to define a Hermitian form over a ring $R$ with involution
$\sigma$ is to give a finitely generated projective right $R$-module $M$ together
with an $R-R$-bimodule map  
\[
b :  M \otimes_{\Z}  M \rightarrow R
\]
where we view $R$ as a
bimodule over itself by $r_1 \cdot r \cdot r_2 = r_1rr_2$, $M$ as
a left $R$-module via the involution, and such that $b(x,y) = \sigma(b(y,x))$. If we remove the final
condition, we obtain a sesquilinear form. 
\end{remark}

By the usual duality, we have a third definition:

\begin{definition}
A
Hermitian module over a ring $R$ with involution is a finitely
generated projective right $R$-module $M$ together with an $R$-linear map $b : M
\rightarrow \overline{M}\dual = M^*$ such that $b = b^*can_M$, where
$b^* : M^{**} \rightarrow M^*$ is given by $(b^*(f))(m) = f(b(m))$. Here 
$can_M : M \rightarrow M^{**}$, $can_M(m)(f) = \overline{f(m)}$ is the canonical 
double dual isomorphism.
\end{definition}

Now, we generalize the above definitions to schemes.

\begin{definition}
Let $X$ be a scheme, and $M$ a quasi-coherent $\mc O_X$-module. Define $\mc O_X\dual
= \iHom_{mod-\mc O_X}(M,\mc O_X)$.
\end{definition}

\begin{definition}
Let $X$ be a scheme with involution $\sigma$, and $M$ a right $\mc
O_X$-module. Note that there's an induced map $\sigma^{\#} : \mc O_X
\rightarrow \sigma_*\mc O_X$. Define the right (note that we're
working with sheaves of commutative rings, so we can do this) $\mc O_X$-module $\overline
M$ to be $\sigma_*M$ with $\mc O_X$ action induced by the map
$\sigma^{\#}$. That is, if $m \in \sigma_*M(U)$, and $c \in \mc
O_X(U)$, then $m \cdot c = m \cdot \sigma^{\#}(c)$. Note that this
last product is defined, because $m \in \sigma_*M(U) =
M(\sigma^{-1}(U))$, $c \in \sigma_*\mc O_X(U) = \mc
O_X(\sigma^{-1}(U))$, and $M$ is a right $\mc O_X$-module. 
\end{definition}

\begin{remark}
We have two choices for the definition of the dual $M^*$. We can
either define 
\[
M^* = \iHom_{mod-\mc O_X}(\sigma_*M,\mc O_X),
\] 
or we can
define $M^* = \sigma_* \iHom_{mod-\mc O_X}(M,\mc O_X)$. We claim that
these two choices of dual are naturally isomorphic. 
\end{remark}

\begin{proof}
Let $f : \sigma_*M|_U \rightarrow \mc O_X|_U$ be a map of right $\mc
O_X|_U$-modules. Post-composing with the map $\mc O_X|_U \rightarrow
\sigma_* \mc O_X|_U$ yields a map $\overline f : \sigma_*M|_U
\rightarrow \sigma_*\mc O_X|_U$, a.k.a. a map $M|_{\sigma^{-1}U}
\rightarrow \mc O_X|_{\sigma^{-1}U}$. Note that \\$\sigma_*\iHom_{mod-\mc
  O_X}(M,\mc O_X)(U) = \iHom_{mod-\mc
  O_X}(M,\mc O_X)(\sigma^{-1}U)$, so that $\overline f \in \sigma_*\iHom_{mod-\mc
  O_X}(M,\mc O_X)(U)$.

On the other hand, given $g \in \sigma_*\iHom_{mod-\mc O_X}(M,\mc
O_X)(U)$, so that $g : \sigma_*M|_U
\rightarrow \sigma_*\mc O_X|_U$, we can postcompose with
$\sigma_*(\sigma^\#)$ to get a map $\widetilde g : \sigma_*M|_*
\rightarrow \sigma_*\sigma_*\mc O_X|_U = \mc O_X|_U$. Since $\sigma^2
= id$, this is clearly the inverse to the map above. 

It's clear that these assignments are natural, since they're just
postcomposition with a natural transformation. 
\end{proof}

\begin{definition}
Define the adjoint module $M^*$ to be $\iHom_{mod-\mc O_X}(\sigma_*M,\mc
O_X)$. By the remark above, it doesn't really matter which of the two
possible definitions we choose here. From this point, we will also use $\iHom_{mod-\mc O_X}$
synonymously with $\iHom_{\mc O_X}$.
\end{definition}

\begin{definition}\label{def:double_dual_iso}
Given a right $\mc O_X$-module $M$, we define the double dual
isomorphism $\can_M : M \rightarrow M^{**}$ as follows: given an open $U
\subseteq X$, we define a map
\[
M(U) \rightarrow \Hom_{\mc O_U}(\sigma_* \iHom_{\mc O_X}(\sigma_* M,\mc O_X)|_U, \mc O_U)
\]
\[
  = \Hom_{\mc O_U}(\iHom_{\mc O_X}(\sigma_* M|_{\sigma(U)}, \mc O_X|_{\sigma(U)}),\mc O_U)
\]
by $u \mapsto \eta_u$, where for an open $V \subseteq U$,
\[
(\eta_u)_V(\gamma) = (\sigma^\#)_{V}^{-1}(\gamma_{\sigma(V)}(u|_V)).
\]

Here $\gamma \in \Hom_{\mc O_{\sigma(U)}}(\sigma_* M|_{\sigma(U)}, \mc
O_{\sigma(U)})$ and $\sigma^\#$ is the morphism of sheaves
$\sigma^\# : \mc O_X \rightarrow \sigma_* \mc O_X$. Note that $\gamma_{\sigma(V)}(u|_V)$ makes sense
because $\sigma_*M(\sigma(V)) = M(V)$.

More globally, there's an evaluation map
\[
ev_\sigma : M \otimes \sigma_* \iHom_{\mc O_X}(\sigma_*M,\mc O_X) \rightarrow \mc O_X
\]
defined by the composition
\[
M \otimes \sigma_* \iHom_{\mc O_X}(\sigma_*M,\mc O_X) \cong M \otimes
\iHom_{\mc O_X}(M,\sigma_*\mc O_X) \xrightarrow{ ev} \sigma_* \mc O_X
\xrightarrow{(\sigma^\#)^{-1}} \mc O_X
\]
which under adjunction yields the above map. 
\end{definition}

\begin{definition}\label{def:herm_form_sch}
Let $X$ be a scheme with involution $- : X \rightarrow X$. Let
$\can_X$ be the double dual isomorphism of Definition \ref{def:double_dual_iso}. A
\emph{Hermitian vector bundle over $X$} is a locally free right $\mc
O_X$-module $V$ with an $\mc O_X$-module map $\phi : V \rightarrow
V^*$ such that $\phi = \phi^*\can_V$. A Hermitian vector bundle is \emph{non-degenerate}
if $\phi$ is an isomorphism.
\end{definition}

\begin{remark}
  Recall that there's an equivalence of categories between locally
  free coherent sheaves on $X$ and geometric vector bundles given by
  $M \mapsto \mathbf{Spec}(\mathrm{Sym}(M\dual))$ in one direction and the sheaf
  of sections in the other. For locally free sheaves, we have $M\dual
  \otimes N\dual \cong (M \otimes N)\dual$ so that the functor is
  monoidal. We will use this to think of a Hermitian
  form as a map of schemes $V \otimes V \rightarrow \mbb A^1$.
\end{remark}

Below we give the key example of a Hermitian vector bundle.

\begin{example}\label{ex:hyp_space}

Define (diagonal) hyperbolic $n$-space over a scheme $(S,-)$ with involution to be $\mbb A^{2n}_S$
with the Hermitian form $(x_1,\dots,x_{2n},y_1,\dots,y_{2n}) \mapsto
\sum_{i=1}^n  \overline x_{2i-1} y_{2i-1} -  \overline x_{2i}
y_{2i}$. Denote this Hermitian form by $h_{\mathrm{diag}}$.

As defined this way, the matrix of this Hermitian form is
\[
\begin{bmatrix}
1 & 0 & \cdots \\
0 & -1 & 0 & \cdots \\
\vdots & \vdots & \vdots \\
0 & \cdots & \cdots & \cdots & -1
\end{bmatrix}
\]
the diagonal matrix diag$(1,-1,1,\dots,-1)$. For this definition to
give a Hermitian space isometric to other standard definitions of the
hyperbolic form, it's crucial that 2 be invertible. 

The isometries of $\mbb H_{\mbb R}$ (where we give it the hyperbolic
form above) have the form
\[
\begin{bmatrix}
a & b\\
\pm b & \pm a
\end{bmatrix}
\]
with $a = \pm\sqrt{1 + b^2}, b \in \mbb R$ (or $a^2 - b^2 = 1$). The
usual identification with $\mbb R^\times \rtimes C_2$ follows by
considering the decomposition $a^2 - b^2 = 1 \iff (a+b)(a-b) = 1$.
\end{example}

\begin{example}
Similarly to above, we can define a hyperbolic form $h$ by the matrix
\[
\begin{bmatrix}
0 & I\\
I & 0
\end{bmatrix}.
\]
This form is isometric to the above form, and we'll use both forms below.
\end{example}

\subsection{Properties}

We record two unsurprising structural results which will be useful when we define
the Hermitian Grassmannian in section \ref{chap:herm_grass}.
\begin{lemma}
Given a map of schemes with involution $f : (Y,i_Y) \rightarrow (X,i_X)$
and a (non-degenerate) Hermitian vector bundle $(V,\omega)$ on $X$, $f^*(V)$ is
a (non-degenerate) Hermitian
vector bundle on $Y$.
\end{lemma}

\begin{proof}
The pullback of a locally free $\mc O_X$-module is a locally free $\mc
O_Y$-module, so we just need to check that it's Hermitian. Given the
map $\omega : V \rightarrow V^*$, we get an induced map $f^*V
\rightarrow f^*(V^*)$ which is an isomorphism if $\omega$ is. Thus we
just need to check that $f^*(V^*) \cong (f^*V)^*$. But pullback
commutes with sheaf dual for locally free sheaves of finite rank, so
we just need to check that changing the module structure via the
involution commutes with pullback; that is, we need to check that
$f^*(\overline V) = \overline{f^*(V)}$. However, this is clear since
the structure map on $f^*(\overline V)$ is given by

\[
O_Y \times f^*V \cong f^* \mc O_X \times f^*V \xrightarrow{f^*(-) \times id} f^*\mc(O_X) \times f^*(V)
\rightarrow f^*(V).
\]
\end{proof}

\begin{lemma}\label{lem:perp_pullback}
Let $(V,\phi)$ be a non-degenerate Hermitian vector bundle over a
scheme with trivial involution $X$, and let $(M,\phi|_M)$ be a (possibly degenerate)
sub-bundle. Given a map of schemes $g : Y \rightarrow X$, there is a
canonical isomorphism
$g^*(M^\perp) \cong (g^*M)^\perp$.
\end{lemma}

\begin{proof}
Recall that, by definition, $M^\perp = \ker(V \xrightarrow{\phi} V^*
\rightarrow M^*)$. Equivalently, $M^\perp$
is defined by the exact sequence
\[
0 \rightarrow M^\perp \rightarrow V \rightarrow M^* \rightarrow 0.
\] 

It follows that the composite map $g^*(M^\perp) \rightarrow g^*V
\rightarrow g^*(M^*)$ is zero, and hence by universal property of kernel there's a canonical map
\[
g^*(M^\perp) \rightarrow \ker(g^*V \rightarrow g^*(M^*) \cong (g^*(M))^*) = (g^*(M))^\perp
\]
 where we've used the canonical isomorphism $g^*(M^*) \cong
 (g^*(M))^*$ for locally free sheaves.

We claim that this map is an isomorphism. It suffices to check on
stalks, where the map can be identified with a map 
\[
M^\perp_{g(y)} \otimes \mc O_{Y,y} \rightarrow \ker(V_{g(y)} \otimes
\mc O_{Y,y} \rightarrow M^*_{g(y)} \otimes \mc O_{Y,y}).
\]

But $V_{g(y)} \cong M^\perp_{g(y)} \oplus M^*_{g(y)}$, so the sequence
\[
0 \rightarrow M^\perp_{g(y)} \otimes \mc O_{Y,y}\rightarrow V_{g(y)}
\otimes \mc O_{Y,y} \rightarrow M^*_{g(y)} \otimes \mc O_{Y,y}
\rightarrow 0
\]
is split exact, and the canonical map is an isomorphism. 
\end{proof}

We record two incredibly useful results for working with Hermitian forms. The first implies that
Hermitian forms over fields can be written as an orthogonal sum of rank 1 Hermitian forms, while the second
gives a useful characterization of non-degenerate submodules of a Hermitian module.

\begin{theorem}\label{thm:field_diag} (Knus \cite{HermKnus} 6.2.4)
Let $(M,b)$ be a non-degenerate Hermitian vector bundle over a division ring
$D$. Then $(M,b)$ has an orthogonal basis in the following cases:
\begin{enumerate}
\item the involution of $D$ is not trivial
\item the involution of $D$ is trivial and
  char $D \neq 2$.
\end{enumerate}
\end{theorem}

\begin{lemma} (Knus)
Let $(M,b)$ be a Hermitian module, and $(U,b|_U)$ be a non-degenerate
finitely generated projective Hermitian submodule. Then $M = U \oplus U^\perp$.
\end{lemma}

\begin{proof}
Since $b|_U : U \rightarrow U^*$ is an isomorphism, given an $m \in
M$, there exists $u \in U$ such that $b(m,-)|_U = b(u,-)|_U$. But then
$b(m-u,-)|_U = 0$, so that $m-u \in U^\perp$, and $m = u + m-u$. Thus
$M = U + U^\perp$. Since $\phi|_U$ is non-degenerate, $U \cap U^\perp
= 0$, so we're done. 
\end{proof}

\subsection{Hermitian Forms on Semilocal Rings}\label{sec:herm_form_semiloc}

From here on out, all rings are assumed to be commutative. Many of the results of this 
section can be deduced from \cite{first}, though we include proofs in an effort to 
make the document self contained.

The following lemma is a slight generalization of a result 
from \cite{Baeza} which will allow us to conclude that
Hermitian forms diagonalize over semilocal rings with involution.

\begin{lemma}\label{lem:semiloc_quot_form}
Let $(R,\sigma)$ be a commutative ring with involution, and let $E$ be a Hermitian module over
$R$. Let $I \subset Jac(R)$ be an ideal fixed by the involution. For every orthogonal decomposition $\overline E = \overline F
\perp \overline G$ of $\overline E = E/IE$ over $R/I$, where $\overline
F$ is a free non-degenerate subspace of $\overline E$, there exists an
orthogonal decomposition $E = F \perp G$ of $E$ with $F$ free and
non-degenerate, and $F/IF = \overline F, G/IG = \overline G$. Here $R/I$ has the 
induced involution $\sigma(x + I) = \sigma(x) + \sigma(I) = \sigma(x) + I$.
\end{lemma}

\begin{proof}
Write $\overline F = \langle \overline x_1 \rangle \oplus \dots \oplus
\langle \overline x_n \rangle$ with $\overline x_i \in \overline F$
and $\det(\overline b(\overline x_i,\overline x_j)) \in (R/I)^\times$. Choose
representatives $x_i \in E$ of $\overline x_i$, and let $F = Rx_1 +
\dots + Rx_n$. We claim that the $x_i$ are independent, so that $F$ is
free: indeed, if $\lambda_1x_1 + \dots + \lambda_nx_n = 0$, then we
get $n$ equations $\lambda_1b(x_1,x_i) + \dots + \lambda_nb(x_n,x_i) =
0$. We claim that $\det(b(x_i,x_j)) = t\in R^\times$. To wit, since $1-st \in I$
for some $s$ by assumption (because the determinant is a unit mod the
ideal $I$), then $st$ cannot be contained in any
maximal ideal, so $st \in R^\times
\implies t \in R^\times$. It follows that the $\lambda_i$ are zero
(otherwise we would have a non-zero vector in the kernel of an
invertible matrix), so that
the $x_i$ are independent as desired. The determinant fact also shows
that $F$ is non-degenerate, so by the lemma above, it has an orthogonal
summand $G$. By construction $F/I = \overline F$, so that $\overline G
= (\overline
F)^{\perp} = (F/I)^{\perp} = F^{\perp}/I = G/I$.
\end{proof}

\begin{lemma}\label{lem:prod_herm}
Hermitian forms over $R_1 \times R_2$ (with trivial involution) are in bijection with $Herm(R_1)
\times Herm(R_2)$.
\end{lemma}

\begin{proof}
First, recall that modules over $R_1 \times R_2$ correspond to a
module over $R_1$ and a module over $R_2$. Indeed, consider the
standard idempotents $(1,0) = e_1, (0,1) = e_2$. Fix a module $M$ over
$R_1 \times R_2$. Then $M = e_1M \oplus e_2M$. Now any $m \in M$
can be written as $e_1m + e_2m = (e_1+e_2)m = m$. Furthermore, if
$e_1m_1 = e_2m_2$, then $e_2e_1m_1=e_2e_2m_2 \implies 0 = e_2m_2$. 

A Hermitian form $M \otimes M \rightarrow R_1 \times R_2$ is
determined by two maps $M \otimes M \rightarrow R_1$ and $M \otimes M
\rightarrow R_2$. Writing $M = e_1M \oplus e_2M$, we note that, by
linearity, it must be the case that $e_1M \otimes e_2M \rightarrow R_1
\times R_2$ is the zero map; to wit, $b(e_1m_1,e_2m_2) =
e_1e_2b(m_1,m_2) = 0$. Thus this Hermitian form is determined
completely by the maps $e_1M \otimes e_1M \rightarrow R_1 \times R_2$
and $e_2M \otimes e_2M \rightarrow R_1 \times R_2$. Finally, note
that, again by linearity, we see that $e_1M \otimes e_1M \rightarrow
R_2$ is the zero map: $b(e_1m_1,e_1m_2) = b(e_1^2m_1,e_1m_2) =
e_1b(e_1m_1,e_1m_2)$, and $e_1R_2 = 0$. Similarly for the other
map. Hence the Hermitian form is completely
determined by the maps $e_1M \otimes e_1M \rightarrow R_1$ and $e_2M
\otimes e_2M \rightarrow R_2$. 
\end{proof}

\begin{corollary}\label{cor:ConstRankTrivInv}
Hermitian modules of constant rank diagonalize over commutative rings $R$ with finitely many maximal
ideals $m_1,\dots,m_n$ (semi-local rings) and with $\frac{1}{2} \in R$.
\end{corollary}

\begin{proof}
By the Chinese Remainder Theorem, $R/(m_1 \cap \dots \cap m_n) \cong
R/m_1 \times \cdots \times R/m_n = F_1 \times \dots \times F_n$. We
claim that Hermitian forms over finite products of fields
diagonalize, and then the result will follow from Lemma \ref{lem:semiloc_quot_form}. By induction and Lemma
\ref{lem:prod_herm}, a Hermitian module $M$
is determined by Hermitian modules $M_i$ over $F_i$, $i =
1,\dots,n$ as $M = M_1 \oplus M_2 \oplus \cdots \oplus M_n$ with
action $(f_1,\dots,f_n) \cdot (m_1,\cdots,m_n) =
(f_1m_1,\dots,f_nm_n)$. By Theorem \ref{thm:field_diag}, each $M_i$ can be diagonalized into $M_i = \langle a_{1,i}
\rangle \perp \dots \perp \langle a_{m,i}\rangle$ (it's important to
note here that the rank of each $M_i$ is the same by assumption). Thus a
diagonalization of $M$ is given by $\langle (a_{1,1}, \dots  ,a_{1,n})\rangle
 \perp \dots \perp \langle (a_{m,1},
\dots ,a_{m,n})\rangle$. 
\end{proof}

\begin{corollary}
Let $R$ be a local ring with trivial involution and with $\frac{1}{2} \in R$. Then any Hermitian module (which is necessarily free)
over $R$ diagonalizes. 
\end{corollary}

\begin{lemma}\label{lem:hyp_switch}
Let $R$ be a ring, and consider the ring $R \times R$ with the
involution that switches factors. Then any module $M$ can be
written as $e_1M \oplus e_2M$ as in the proof of Lemma \ref{lem:prod_herm}. A non-degenerate Hermitian form on this
module is determined by a map $e_1M \otimes e_2M \rightarrow R \times
R$. In other words, the matrix representing the map $e_1M \oplus e_2M
\rightarrow e_1M^* \oplus e_2M^*$ has the form
\[
\begin{bmatrix}
0 & A \\
\overline{A}^t & 0
\end{bmatrix}.
\]

where $A$ is invertible. 
\end{lemma}

\begin{proof}
The first claim is just that $b(e_1x,e_1y) = 0 = b(e_2x,e_2y)$ for any
$x,y \in M$. This follows because $b(e_1x,e_1y) = b(e_1^2x,e_1^2y) =
\overline{e_1}e_1b(e_1x,e_1y) = e_2e_1b(e_1x,e_1y) = 0$. Similarly for
$b(e_2x,e_2y)$. The statement about the matrix follows by identifying
the map $M \otimes \overline M \rightarrow R \times R$ with an isomorphism $M
\rightarrow \overline M^*$ and using the direct sum decomposition. 
\end{proof}

\begin{corollary} \label{cor:HermFormsOverSwitch} 
Let $R,M$ be as in lemma \ref{lem:hyp_switch} and such that $\frac{1}{2} \in R$. 
Then $M \cong H(e_1M)$, where $H$ denotes
the hyperbolic module functor. 
\end{corollary}

\begin{proof}
The assumption that 2 is invertible implies that $M$ is an even
Hermitian space in the notation of Knus. Now by Lemma \ref{lem:hyp_switch}
$b|_{e_1M} = 0$, so $M$ has direct summands $e_1M, e_2M$ such that $e_1M =
e_1M^{\perp}$ and $M = e_1 M \oplus e_2M$. Now \cite[Corollary 3.7.3]{HermKnus} 
applies to finish the proof.
\end{proof}

\begin{corollary}\label{cor:herm_diag}
Let $R$ be a semi-local ring with involution and with 2 invertible. Then any Hermitian
module of constant rank over $R$ diagonalizes.
\end{corollary}

\begin{proof}
Using Lemma \ref{lem:semiloc_quot_form} and reducing modulo the Jacobson radical
(which is always stable under the involution), it
suffices to prove the corollary for $R$ a finite product of
fields. Then $R = F_1 \times \dots \times F_n$ is semi-simple, and
hence we can index the fields in a particularly nice way (proof is by
considering idempotents), writing $R = A_1 \times \dots \times A_m \times B_1 \times
\dots B_{n-m}$ such that $A_i$ is fixed set-wise by the involution, and
$\sigma(B_{2i}) = B_{2i+1}$, $\sigma(B_{2i+1}) = B_{2i}$. Now, any
finitely generated module $M$ can be written as a direct sum $M =
\bigoplus_{i=1}^m M_i \bigoplus_{i=1}^{\frac{n-m}{2}} N_{2i} \oplus
N_{2i-1}$. By Theorem \ref{thm:field_diag} and Corollary \ref{cor:HermFormsOverSwitch}, the form when restricted to each
$M_i$ or $N_{2i} \oplus N_{2i-1}$ is diagonalizable, so the form is
diagonalizable (see the proof of Corollary \ref{cor:ConstRankTrivInv}).   
\end{proof}

\begin{lemma}\label{lem:semiloc_herm_form_rank}
Non-degenerate Hermitian vector bundles are determined by rank over strictly
henselian local rings $(R,m)$ with $\frac{1}{2} \in R$ such that the residue field $R/m$ has
trivial involution.
\end{lemma}

\begin{proof}
By Corollary \ref{cor:herm_diag}, any Hermitian vector bundle 
 over $R$
diagonalizes. Thus it suffices to prove that any two non-degenerate
Hermitian vector bundles of rank 1 are isometric. 

A non-degenerate rank 1 Hermitian vector bundle corresponds to a unit
$x \in R^\times$ such that $x = \overline x$ (a one dimensional
Hermitian matrix). Because $R$ is strictly henselian, there is a
square root $c$ of $x^{-1}$. We claim that $c = \overline c$. Assume
not. Then because the involution on $R/m$ is trivial, $c - \overline
c \in m$. Since 2 is invertible, we have $c = \frac{c + \overline
  c}{2} + \frac{c-\overline c}{2}$. It follows that $\frac{c+\overline
  c}{2}$ is a unit. Otherwise it would be contained in $m$ which would
imply that the unit $c$ was contained in $m$. 

However, we calculate $(c + \overline c)(c - \overline c) = c^2 -
\overline c^2$. But $(\overline c)^2 = \overline{(c^2)} = \overline x^{-1}
= x^{-1}$, so that $(c + \overline c)(c - \overline c) = 0$. Because $c +
\overline c$ is a unit, it follows that $c - \overline c = 0$. 

This shows that given any one dimensional Hermitian matrix $x$,
there's a unit $c$ such that $c x \overline c = 1$ so that all one
dimensional Hermitian forms are isometric to the form $\langle 1 \rangle$.
\end{proof}

\begin{corollary}\label{cor:loc_det_rank}
Non-degenerate Hermitian vector bundles are locally determined by rank in the isovariant \'etale topology.
\end{corollary}

\begin{proof}
The points in the isovariant \'etale topology are either strictly
henselian local rings whose residue field has
trivial involution or a ring of the form $\mc O_{X,x}^{sh} \times \mc
O_{X,x}^{sh}$ with involution $(x,y) \mapsto (i(y),i(x))$. Via the map
$(x,y) \mapsto (x,i(y))$, such rings are isomorphic to hyperbolic
rings. 

If the ring is a stricty henselian local ring whose residue field has
trivial involution, Lemma \ref{lem:semiloc_herm_form_rank} shows that
non-degenerate Hermitian forms are determined by rank. If the
ring is hyperbolic, then by Corollary \ref{cor:HermFormsOverSwitch}  all non-degenerate Hermitian forms over the
ring are hyperbolic forms of projective modules over a local
ring. Since projective modules over a local ring are determined by
rank, the corresponding hyperbolic forms are determined by rank. 
\end{proof}

\subsection{Higher Grothendieck-Witt Groups}\label{sec:Higher_GW}

In \cite{Xie2018ATM}, the author works with coherent Grothendieck-Witt
groups on a scheme. Because the negative $K$-theory of the category of
bounded complexes of quasi-coherent $\mc O_X$-modules with coherent
cohomology vanishes (together with the pullback square relating the
homotopy fixed points of $K$-theory to Grothendieck-Witt theory), there is no difference between the additive and
localizing versions of Grothendieck-Witt spectra in this
setting. 

Therefore, we work instead with Grothendieck-Witt spectra of
$\sPerf(X) = \mathrm{Ch}^b\Vect(X)$, the dg category of strictly
perfect complexes on $X$. We review the relevant definitions from
\cite{Schder} now.

\begin{definition}
A \emph{pointed dg category with duality} is a triple $(\mc
A,\vee,\can)$ where $\mc A$ is a pointed dg category, $\vee : \mc
A^{op} \rightarrow \mc A$ is a dg functor called the duality functor,
and $\can : 1 \rightarrow \vee \circ \vee^{op}$ is a natural
transformation of dg functors called the double dual identification
such that $\can_A^\vee \circ \can_{A^\vee} = 1_{A^\vee}$ for all
objects $A$ in $\mc A$. 
\end{definition}

\begin{remark}
A dg category with duality has an underlying
exact category with duality $(Z^0\mc
A^{\ptr},\vee,\can)$, where $Z^0\mc A^{\ptr}$ has the same objects
as $\mc A^{\ptr}$ but the morphism sets are the zero cycles in the morphism
complexes of $\mc A^\ptr$. Here $\mc A^\ptr$ is the pretriangulated hull
of $\mc A$ (see \cite{Schder} definition 1.7).
\end{remark}

\begin{definition}
A \emph{dg category with weak equivalences} is a pair $(\mc A,w)$
where $\mc A$ is a pointed dg category and $w \subseteq Z^0 \mc
A^\ptr$ is a set of morphisms which saturated in $\mc A$. A map $f$ in
$w$ is called a weak equivalence.
\end{definition}

\begin{definition}
Given a pointed dg category with duality $(\mc A,\vee,\can)$, a
Hermitian object in $\mc A$ is a pair $(X,\phi)$ where $\phi : X
\rightarrow X^{\vee}$ is a morphism in $\mc A$ satisfying
$\phi^{\vee}\can_X = \phi$. 
\end{definition}

\begin{definition}
A \emph{dg category with weak equivalences and duality} is a quadruple
$\mathscr A = (\mc A, w, \vee, \can)$ where $(\mc A,w)$ is a dg
category with weak equivalences and $(\mc A,\vee,\can)$ is a dg
category with duality such that the dg subcategory $\mc A^w \subset
\mc A$ of $w$-acyclic objects is closed under the duality functor
$\vee$ and $\can_A : A \rightarrow A^{\vee\vee}$ is a weak equivalence
for all objects $A$ of $\mc A$. 
\end{definition}

\begin{definition}
Let $\mathscr A = (\mc A,w,\vee,\can)$ be a dg category with weak
equivalences and dualiy. A symmetric space in $\mathscr A^{\mathrm{ptr}}$ 
is a Hermitian object $A$ whose dual map $\phi : A \rightarrow A^\vee$ is
a weak equivalence in $\mathscr A^{\mathrm{ptr}}$. 
The Grothendieck-Witt group $GW_0(\mathscr
A)$ of $\mathscr A$ is the abelian group generated by symmetric
spaces
$[X,\phi]$ in the underlying category with weak equivalences and
duality $(Z^0\mc A^{ptr},w,\vee,\can)$, subject to the following
relations:
\begin{enumerate}
\item $[X,\phi] + [Y,\psi] = [X \oplus Y,\phi \oplus \psi]$
\item if $g : X \rightarrow Y$ is a weak equivalence, then $[Y,\psi] =
  [X,g^\vee \psi g]$, and
\item if $(E_\bullet,\phi_\bullet)$ is a symmetric space in the
  category of exact sequences in $Z^0\mc A^\ptr$, that is, a map
\[
\xymatrix{E_\bullet :  \ar[d]_\sim^{\phi_\bullet} \\ E^\vee_\bullet :}
\qquad \xymatrix{E_{-1} \, \ar@{>->}[r]^i \ar[d]_\sim^{\phi_{-1}} &
  E_0 \ar@{->>}[r]^p \ar[d]_\sim^{\phi_0} & E_1 \ar[d]_\sim^{\phi_1}
  \\
E_{1}^\vee \, \ar@{>->}[r]^{p^\vee} &
  E_0^\vee \ar@{->>}[r]^{i^\vee} & E_1^\vee }
\]
of exact sequences with $(\phi_{-1},\phi_0,\phi_1) = (\phi_1^\vee
\can, \phi_0^\vee \can, \phi^\vee_{-1}\can)$ a weak equivalence, then
\[
[E_0,\phi_0] = \left[E_{-1} \oplus E_1, \begin{pmatrix}
0 & \phi_1\\
\phi_{-1} & 0
\end{pmatrix} \right].
\]
\end{enumerate}
\end{definition}

\begin{definition}\label{def:GW_spectra}
Given a dg-category with weak equivalences and duality $\mathscr A = (\mc A,w,\vee,\can)$, Schlichting
defines \cite[Section 4.1]{Schder}  a functorial monoidal symmetric spectrum
$GW(\mathscr A)$ using a modified version of the Waldhausen $\mc
S_\bullet$ construction. For the sake of brevity, we don't reproduce
his construction here. 

Noting in general that $GW$ doesn't sit in a localization sequence,
Schlichting defines a localizing variant, $\mbb GW$ in \cite[Section
8.1]{Schder} as a bispectrum. The reason Schlichting defines
$\mbb GW$ as an object in bispectra rather than spectra is to get a monoidal structure
on $\mbb GW$. We provide an alternative approach to producing $\mbb
GW$ via periodization in section \ref{chap:Einf}. 
\end{definition}

\begin{definition}
Let $X$ be a Noetherian scheme of finite Krull dimension with an ample
family of line bundles, and let
$\sigma : X \rightarrow X$ be an involution on $X$. Let $\sPerf(X)$
denote the category of strictly perfect complexes on $X$ with the weak
equivalences being the quasi-isomorphisms. Define a family of dualities on
$\sPerf(X)$ indexed by $i \in \mbb N$ by 
\[
\ast^i : E \mapsto \iHom_{\sPerf(X)}(\sigma_*E,\mc O_X[i]).
\]
Note that because $\sigma$ is an involution, $\sigma_*E$ is a
strictly perfect complex. Define the canonical isomorphim $\can$ as in Definition
\ref{def:double_dual_iso} as the adjoint of the evaluation map
\[
ev : E \otimes \sigma_*\iHom_{\sPerf(X)}(\sigma_* E, \mc O_X[i])
\rightarrow \mc O_X[i].
\]
Combining all this data we get a collection of dg categories with weak equivalences
and duality 
\[(\sPerf(X),\text{q. iso}, \ast^i,\can).
\]

The $i$th shifted Grothendieck-Witt spectrum of $(X,\sigma)$ is
defined as
\[
GW^{[i]}(X,\sigma) = GW (\sPerf(X),\text{q. iso}, \ast^i,\can).
\]

If $Z$ is an invariant closed subset of $X$, then the duality on
$\sPerf(X)$ restricts to a duality on the subcategory of complexes
supported on $Z$, $\sPerf_Z(X)$. We define 
\[
GW^{[i]}(X \text{ on } Z) = GW (\sPerf_Z(X),\text{q. iso}, \ast^i,\can).
\]
\end{definition}

\section{Representability of Automorphism Groups of Hermitian Forms}\label{chap:herm_grass}

Representability of $K$-theory in the stable motivic homotopy category
allows one to check that $K$-theory pulls back nicely. In particular,
given $f : X \rightarrow S$ a map of schemes over $S$, one can use
ind-representability of $\KGL$ to show that
$f^*(\KGL_S) = \KGL_X$. Together with the formalism of six operations
in motivic homotopy theory, one obtains rather formally cdh descent for algebraic
$K$-theory, see \cite{DenRep}.

The goal of
this section is to define a sheaf on $\Sm{S,qp}^{C_2}$, denoted $\RGr$, which represents
Hermitian $K$-theory in the motivic homotopy category $\mathcal
\mothom_S^{C_2}$. We first check that over a regular base $S$ with 2 invertible
(e.g. $\Z[\frac{1}{2}]$), Hermitian $K$-theory is representable in the
category of $C_2$-schemes over $S$, $\Sm{S,qp}^{C_2}$. To
extend this result to non-regular bases $S$, we utilize the
Morel-Voevodsky approach to classifying spaces and obtain
representability of homotopy Hermitian $K$-theory in the motivic homotopy category $\mathcal
\mothom_S^{C_2}$.

By analogy with the $K$ theory case, the equivariant scheme
representing Hermitian $K$-theory on $\Sm{S,qp}^{C_2}$ will be a colimit of schemes
which parametrize non-degenerate Hermitian sub-bundles of a given
Hermitian vector bundle $V$. The new results here are mostly the
definitions, as the proofs in this section are either
minor modifications or identical to the proofs in \cite{SchTri}. The
main difference which might cause concern is that stalks in the
isovariant \'etale topology are now semi-local (rather than
local) rings. 

We combine the techniques of
\cite{SchTri} with a Morel-Voevodsky style argument to compare $\RGr_{2d}(\mbb H^\infty)$ to the isovariant \'etale
classifying space $B_{isoEt}O(\mbb H^{d})$ of the group of automorphisms of
hyperbolic $d$-space. The key to the comparison is that locally in the
isovariant \'etale topology, Hermitian vector bundles are determined
by rank. This will utilize some of the analysis of
Hermitian forms over semi-local rings from section
\ref{sec:herm_form_semiloc}. Note that this is a key difference from
the $K$-theory case where one must pass only to local (rather than
strictly henselian local) rings in order for $K$-theory to be
determined by rank. 

A straightforward generalization of the techniques in \cite{SchTri} allows one to compare $\colim_n B_{isoEt}O(\mbb H^n)(\Delta R)$ to
the Grothendieck-Witt space defined in section \ref{sec:Higher_GW} by viewing them both as group completions
and comparing their homology. This approach is inspired by the
Karoubi-Villamayor definition of higher algebraic $K$-theory. We don't carry out this comparison here as it is 
unnecessary for proving cdh descent. 

\subsection{The definition of the Hermitian Grassmannian $\RGr$}

The definition here describes the sections of the underlying scheme of
$\RGr$ over a scheme $X \rightarrow S$. We advise the hurried reader to
skip to section \ref{subsec:representability}.

\begin{lemma}\label{lem:action_presheaf}
Let $\mc F$ be a presheaf on $\Sm{S,qp}$ and let $a : \mc F \implies
\mc F$ be a natural transformation such that $a \circ a = id_{\mc F}$. Then
there's an associated presheaf on $\Sm{S,qp}^{C_2}$ defined by the formula
$(X,\sigma : X \rightarrow X) \mapsto \mc F(X)^{C_2}$ where the action
of $C_2$ on $\mc F(X)$ is defined by $f \mapsto  a_X\mc F(\sigma)(f)$.
\end{lemma}

\begin{proof}
Note that this is indeed a $C_2$-action, since $a_X \mc F(\sigma)(
a_X\mc F(\sigma)(f)) = \mc F(\sigma) a_X(a_X \mc F(\sigma)(f)) = \mc
F(\sigma)(\mc F(\sigma)(f)) = f$ using naturality. 
\end{proof}

Fix a (possibly degenerate) Hermitian vector bundle $(V,\phi)$ over a
base scheme $S$ with 2 invertible and with trivial
involution. The canonical example of such a base scheme is $S = \Spec \Z[\frac{1}{2}]$.

We'll define a presheaf $\RGr : (\Sm{S,qp}^{C_2})^{op} \rightarrow \Set$
by first defining a presheaf on $\Sm{S,qp}$, showing that it's
representable, equipping with an action, then taking the corresponding
representable functor on $\Sm{S,qp}^{C_2}$. We can then extend to an
arbitrary equivariant base $T$ with 2 invertible by pulling back along the unique map $T
\rightarrow \Z[\frac{1}{2}]$.

\begin{itemize}
\item On objects, $\RGr(V)(f: X \rightarrow S)$ for an $S$-scheme $f :
  X \rightarrow S$ is a split surjection $(p,s)$
\[
\xymatrix{f^*V \ar@{>>}[r]_p & \ar@/_1.0pc/@{-->}[l]_sW},
\]

where $W$ is locally free. 

Here by an isomorphism of split surjections we mean a diagram
\[
\xymatrix{f^*V \ar@{>>}[r]_p \ar@{=}[d]& \ar@/_1.0pc/@{-->}[l]_sW \ar[d]^\phi\\
f^*V \ar@{>>}[r]_{p'} & \ar@/_1.0pc/@{-->}[l]_{s'}W'}
\]

such that $\phi$ is an isomorphism satisfying $\phi \circ p = p'$ and $s = s'\circ \phi$.

\item Given a morphism 
\[
\xymatrix{Y\ar[dr]_h \ar[rr]^g &&X\ar[dl]^f\\& S&}
\]
over $S$, define
\[
\RGr_V(g)(\xymatrix{f^*V \ar@{>>}[r]_p & \ar@/_1.0pc/@{-->}[l]^sW}) = 
\xymatrix{h^*V \ar[r]^{can} & g^*f^*V \ar@{>>}[r]_{g^*p} &
  \ar@/_1.0pc/@{-->}[l]_{g^*s} g^*W}.
\]
\end{itemize}
There's a natural action of $C_2$ on $\RGr_V$ whose non-trivial
natural transformation will be denoted $\eta$. Define $\eta$ as
follows: 

 Fix an object $X \in \Sm{S,qp}$. Define
\[
\eta_X(\xymatrix{f^*V \ar@{>>}[r]_p & \ar@/_1.0pc/@{-->}[l]_sW}) =
\xymatrix{f^*V \ar@{>>}[r]_q & \ar@/_1.0pc/@{-->}[l]_t(\ker p)^\perp}.
\]

We'll define the maps $q$ and $t$ now. Let $t'$ denote the canonical map
$\ker p \rightarrow f^*V$, let $q'$ be the map
\[
q' : \xymatrix{f^*V \ar[r]^{\mathrm{id} - (s \circ p)} & f^*V \ar[r]^{\im} & \im(t') \ar[r]^{(t')^{-1}} & \ker p}  
\]
where we've used the identification $\im t' = \im (\mathrm{id} - (s \circ p))$.


Recall that 
\[
W^\perp = \ker(f^*V \xrightarrow{f^*\phi} f^*(V^*) \xrightarrow{can}
(f^*V)^* \xrightarrow{s^*} W^*)
\]
and similarly for $(\ker p)^\perp$.

Leaving out the $can$ map for convenience, we get split exact sequences
\[
\xymatrix{0 \ar[r] & W^\perp \ar[r] & f^*V \ar[r]_{s^*}& \ar@/_1.0pc/[l]_{p^*} W^*
\ar[r] & 0}
\]
and
\[
\xymatrix{0 \ar[r] & (\ker p)^\perp \ar[r] & f^*V \ar[r]_{(t')^*}& \ar@/_1.0pc/[l]_{(q')^*} (\ker p)^*
\ar[r] & 0}.
\]

By the splitting lemma for abelian categories, $f^*V \cong W^\perp
\oplus W^*$, and there's a (canonical) split surjection $f^*V
\twoheadrightarrow W^\perp$ with $W^\perp$ locally free. Similarly we 
obtain a canonical surjection $q : f^*V \twoheadrightarrow (\ker p)^\perp$ 
split by a map $t$.

Given an isomorphism 

\[
\xymatrix{f^*V \ar@{>>}[r]_p \ar@{=}[d]& \ar@/_1.0pc/@{-->}[l]_sW \ar[d]^\psi\\
f^*V \ar@{>>}[r]_{p'} & \ar@/_1.0pc/@{-->}[l]_{s'}W'}
\]
we get an isomorphism of (split) diagrams
\[
\xymatrix{f^*V \ar[r]^{f^*\phi}\ar@{=}[d] &(f^*V)^*\ar[r]^{s^*} \ar@{=}[d] & W^* \ar[d]^{(\psi^{-1})^*}\\
f^*V \ar[r]^{f^*\phi} & (f^*V)^* \ar[r]_{(s')^* }& (W')^*}
\]
and hence an isomorphism of split surjections
\[
\xymatrix{f^*V \ar@{>>}[r]_q \ar@{=}[d]& \ar@/_1.0pc/@{-->}[l]_tW^\perp \ar[d]^\delta\\
f^*V \ar@{>>}[r]_{q'} & \ar@/_1.0pc/@{-->}[l]_{t'}(W')^\perp},
\]

so that $\eta_X$ is a well-defined map of sets. Given a map of schemes
$g : Y \rightarrow X$, such that $f\circ g = h$ and an element 
\[
\xymatrix{f^*V \ar@{>>}[r]_p & \ar@/_1.0pc/@{-->}[l]_sW}
\]

in $\RGr_V(X)$, 
\begin{align*}
\RGr(g) \circ \eta_X(\xymatrix{f^*V \ar@{>>}[r]_p &
  \ar@/_1.0pc/@{-->}[l]_sW}) = & \RGr(g)( \xymatrix{f^*V \ar@{>>}[r]_q
  & \ar@/_1.0pc/@{-->}[l]_t(\ker p)^\perp})\\
 = & \xymatrix{h^*V \ar[r]^{can} & g^*f^*V \ar@{>>}[r]_{g^*q} &
  \ar@/_1.0pc/@{-->}[l]_{g^*t} g^*((\ker(p))^\perp)}
\end{align*}

while
\begin{align*}
\eta_Y\circ\RGr(g)(\xymatrix{f^*V \ar@{>>}[r]_p &
  \ar@/_1.0pc/@{-->}[l]_sW})) = \xymatrix{h^*V \ar[r]^{can} & g^*f^*V \ar@{>>}[r]_{q'} &
  \ar@/_1.0pc/@{-->}[l]_{t'} (g^*(\ker(p)))^\perp}
\end{align*}

By Lemma \ref{lem:perp_pullback}, there's a canonical isomorphism
$g^*((\ker(p)^\perp)) \rightarrow (g^*(\ker(p)))^\perp$, and under
this isomorphism $q'$ and $t'$ correspond to $g^*q$, and $g^*t$,
respectively. This concludes the check of naturality.

Now by Lemma \ref{lem:action_presheaf}, there's a presheaf $\RGr :
\Sm{S,qp}^{C_2} \rightarrow \Set$. To determine its values on a
$C_2$-scheme $(X,\sigma)$, we note that a fixed point of the action of Lemma
\ref{lem:action_presheaf} is determined by an isomorphism of split
surjections

\[
\xymatrix{f^*V \ar@{>>}[r]_q \ar@{=}[d]& \ar@/_1.0pc/@{-->}[l]_t\sigma^*(\ker(p)^\perp) \ar[d]^\psi\\
f^*V \ar@{>>}[r]_{p} & \ar@/_1.0pc/@{-->}[l]_{s}\ker(p)}
\]

Note that because $\sigma$ is an involution, for any $\mc O_X$-module
$M$, there's a canonical isomorphism of $\mc O_X$-modules $\sigma_*M \cong
\sigma^*M$. Thus there's a natural isomorphism 
\[
\Hom_{mod-\mc
  O_X}(\sigma_*f^*V,-) \cong \Hom_{mod-\mc O_X}(\sigma^*f^*V,-) \cong
\Hom_{mod-\mc O_X}(f^*V,-).
\]

It follows that any Hermitian form
\[
\phi : f^*V \rightarrow \Hom_{mod-\mc O_X}(f^*V,\mc O_X)
\]
can be promoted to a Hermitian form
\[
\widetilde\phi : f^*V \rightarrow \Hom_{mod-\mc O_X}(\sigma_*f^*V,\mc O_X)
\]
compatible with an involution $\sigma$ on $X$.

Let $(M,\phi|_M)$ be a Hermitian sub-bundle of $f^*V$ over the scheme $X$ with
trivial involution. We claim that $\sigma^*(M^\perp)$ is the
orthogonal complement of $M$ viewed as a Hermitian sub-bundle of
$f^*V$ with the promoted form $\widetilde \phi$. Said differently, we
claim that
\[
\sigma^*(\ker(f^*V \xrightarrow{\phi|_M} \Hom(M,\mc O_X))) \cong \ker(f^*V
\xrightarrow{\widetilde \phi|_M} \Hom(\sigma_*M,\mc O_X)).
\]

But using the natural isomorphism between $\sigma^*$ and $\sigma_*$,
together with the natural isomorphisms 
\[
\sigma^*\Hom(M,\mc O_x) \cong
\Hom(M,\mc O_X)
\] and $\sigma^*f^*V \cong f^*V$,
this becomes a question
of whether $\sigma^*$ is left exact. In general it isn't, but because
$\sigma$ is an involution, $\sigma^*$ is naturally isomorphic to
$\sigma_*$ which is left exact. The claim follows.


\subsection{Representability of $\RGr$}\label{subsec:representability}

Fix a Hermitian vector bundle $(V,\phi)$ over $S$ where $\dim(V) = n$
and $S$ is a scheme with
trivial involution. Then the underlying scheme of $\RGr(V)$ is the
pullback 
\[
\xymatrix{\RGr(V) \ar[r]\ar[d] & \uline{\Hom}_{\mc O_S}(V,V) \times
  \uline{\Hom}_{\mc O_S}(V,V) \ar[d]^{\circ, id} \\ \uline{\Hom}_{\mc
    O_S}(V,V) \ar[r]^-\Delta & \uline{\Hom}_{\mc O_S}(V,V) \times \uline{\Hom}_{\mc O_S}(V,V)}
\]
where the right vertical map sends $p \mapsto (p \circ p, p)$. In
other words, the underlying scheme is the scheme of idempotent
endomorphisms of $V$. The action corresponds to the map $p \mapsto
p^\dagger$, where $p^\dagger$ is the adjoint of $p$ with respect to
the form $\phi$. 

Note that using this description, an equivariant map $(X,\sigma) \rightarrow
\RGr(V)$ corresponds to an idempotent $p : V_X \rightarrow V_X$ such
that $\phi^{-1}(\gamma^{-1}(\sigma^*p)\gamma)^*\phi = p$, where we're being cavalier
and using $\ast$ to denote both dual (on the outside) and pullback (by
$\sigma$). Here $\gamma$ is the canonical isomorphism $V_X
\xrightarrow{\gamma} \sigma^*V_X$; if the structure map of $X$ is $f :
X \rightarrow S$, then $\gamma$ arises from the equality $\sigma \circ f = f$.

Note that the form on $V_{(X,\sigma)}$ is by definition the composite

\[
\widetilde \phi : V_X \xrightarrow{\phi} V_X^* \xrightarrow{(\gamma^*)^{-1}} \sigma^*V_X^*
\xrightarrow{(\eta^*)^{-1}} \sigma_*V_X^*,
\]
and the adjoint of $p$ is given by $\widetilde \phi^{-1} (\sigma_*p)^*
\widetilde \phi$. Expanding, this is

\[
\phi^{-1}(\gamma^*)(\eta^*)(\eta^*)^{-1}(\sigma^*p)^*(\eta^*)(\eta^*)^{-1}(\gamma^*)^{-1}\phi
= \phi^{-1}(\gamma^{-1}(\sigma^*p)\gamma)^*\phi,
\]

and so we recover the condition that $p^\dagger  =  p$, which
corresponds to the fact that $V_X = \ker p \perp \im p$, and hence the
restriction of the form on $V_X$ to $\im p$ (and $\ker p$) is
non-degenerate. 

To summarize, the underlying scheme of $\RGr(V)$ represents
idempotents, and equivariant maps pick out those idempotents which
correspond to orthogonal projections. 

\begin{definition}

Now fix a dimension $d$ and a non-degenerate Hermitian vector bundle
$(V,\phi)$ over $S$.  Define $\RGr_d(V)$ to be the closed
subscheme of $\RGr(V)$ cut out by $\rk(p) = d$, where $\rk$ is the
rank map. In other words, $\RGr_d(V)$ is the pullback
\[
\xymatrix{\RGr_d(V) \ar[r]\ar[d] & \RGr(V) \ar[d]^{\rk} \\ \{d\}
  \ar[r] & \Z}
\]
The requirement that $V$ be non-degenerate is necessary so that the action
on $\RGr(V)$ sends rank $d$ subspaces to rank $d$ subspaces and hence
induces an action on $\RGr_d(V)$. 
\end{definition}

\begin{remark}
Denote by $g : \RGr_d(V) \rightarrow S$ the structure map of
$\RGr_d(V)$.  Because $\RGr_d(V)$ is representable by a $C_2$-scheme, there's an
idempotent $g^*(V) \rightarrow g^*(V)$ corresponding to the identity map $id: \RGr_d(V) \rightarrow
\RGr_d(V)$. This idempotent is simply the idempotent which, over a
point of $\RGr_d(V)$ represented by an idempotent $p : V \rightarrow
V$, restricts to $p$. There's an action $\sigma$ on $\RGr_d(V) \times_S V$
induced by the action on $\RGr_d(V)$, and using the fact that $\sigma
p \sigma = p^\dagger$ one can see that this idempotent is
non-degenerate with respect to the promoted Hermitian form on $g^*(V)$
compatible with the involution on $\RGr_d(V)$.
\end{remark}

\begin{remark}
Since we've shown that $\RGr(V)$ represents non-degenerate Hermitian
subbundles of $V$, at this point we'll move away from explicitly
referring to split surjections and just represent the sections of
$\RGr(V)$ by non-degenerate subbundles. 
\end{remark}

\begin{definition}
Let $\mbb H_S$ denote the hyperbolic space \ref{ex:hyp_space} over the base scheme $S$. 
Let $\mbb H^\infty = \colim_n \mbb H^n_S$. Similarly given a non-degenerate Hermitian 
vector bundle $V$, let $V \perp \mbb H^\infty = \colim_n V \perp \mbb H^n$. 
For $V \subset \mbb H^\infty$
a constant rank non-degenerate subbundle,
let $|V|$ denote the rank of $V$. Order such subbundles of $\mbb
H^\infty$ by inclusion, and denote the resulting poset by $P$. Given an inclusion $V \hookrightarrow V'$
of non-degenerate subbundles, denote by $V'-V$ the complement of $V$
in $V'$. Let $\mc H : P \rightarrow \mathrm{Fun}(\Sm{S,qp}^{C_2,
  op},\mathrm{Set})$ be the functor which on objects sends a subbundle $V$ to $
\RGr_{|V|} (V \perp \mbb H^\infty) = \colim_n \RGr_{|V|}(V \perp \mbb H^n)$. 
Given an inclusion $V \hookrightarrow V'$, the induced map
$\RGr_{|V|}(V \perp \mbb H^\infty) \rightarrow \RGr_{|V'|}(V' \perp
\mbb H^\infty)$ is given by $E \mapsto E \perp (V'-V)$. Note that
because $V$ is non-degenerate, $V \perp (V'-V) = V'$. Define
\begin{equation}\label{def:RGr_inf}
\RGr_\bullet = \colim \mc H.
\end{equation}

\end{definition}

\subsection{The \'Etale Classifying Space}

The content of this section is a straightforward generalization of the work of \cite{SchTri} to the $C_2$-equivariant setting. Fix a scheme $S$ with 2 invertible, and let $(V,\phi)$ be a
(possibly degenerate) Hermitian vector bundle over $S$. For a
$C_2$-scheme $f : X \rightarrow S$ over $S$, let
\[
\mc S(V,\phi)(X)
\]
be the category of non-degenerate Hermitian sub-bundles of $f^*V$. A
morphism in this category from $E_0$ to $E_1$ is an isometry not
necessarily compatible with the embeddings $E_0,E_1 \subseteq
f^*V$. Using pullbacks of quasi-coherent modules, we turn $\mc S$ into a presheaf of categories on
$\Sm{S,qp}^{C_2}$. For integer $d \geq 0$, define
\[
\mc S_d(V,\phi)  \subset \mc S(V,\phi)
\]
to be the presheaf which on a $C_2$-scheme $f : X \rightarrow S$
assigns the full subcategory of non-degenerate Hermitian sub-bundles
of $(f^*V,f^*\phi)$ which have constant rank $d$. The associated
presheaf of objects is $\RGr_d(V,\phi)$.

Note that the object $V = (V,0) \in \mc S_{|V|}(V \perp \mbb H^\infty)$ has
automorphism group $O(V)$. Thus we get an inclusion $O(V,\phi)
\rightarrow \mc S_{|V|}(V \perp \mbb H^\infty) $, where $O(V)$ is the
isometry group considered as a category on one object. After isovariant \'etale
sheafification, this inclusion becomes an equivalence; this follows
from Corollary \ref{cor:loc_det_rank} that on the points in the isovariant \'etale
topology, Hermitian vector bundles are determined by rank.

Upon applying the nerve, we get maps of simplicial presheaves $BO(V)
\rightarrow B\mc S_{|V|}(V\perp \mbb H^\infty)$ which is a weak
equivalence in the isovariant \'etale topology. Abusing notation, let $B_{isoEt}O(V)$
denote a fibrant replacement of $B\mc S_{|V|}(V\perp \mbb H^\infty)$ in the isovariant
\'etale topology so that we get a sequence of maps
\[
BO(V) \rightarrow B\mc S_{|V|}(V\perp \mbb H^\infty) \rightarrow B_{isoEt}O(V).
\]
which are weak equivalences in the isovariant \'etale topology.

\begin{lemma} \label{lem:weqaff}
Let $(V,\phi)$ be a non-degenerate Hermitian vector bundle over a
scheme $S$ with trivial involution and $\frac{1}{2} \in \Gamma(S,\mc
O_S)$. Then for any affine $C_2$-scheme $\Spec R$ over $S$, the map
\[
B\mc S_{|V|}(V\perp \mbb H^\infty)(R) \rightarrow B_{isoEt}O(V)(R)
\]
is a weak equivalence of simplicial sets. In particular, the map
\[
B\mc S_{|V|}(V\perp \mbb H^\infty) \rightarrow B_{isoEt}O(V)
\]
is a weak equivalence in the equivariant Nisnevich topology, and hence an
equivalence after $C_2$ motivic localization. 
\end{lemma}

\begin{proof}

Each Hermitian vector bundle $W  \in S_{|V|}(V\perp \mbb H^\infty)(R)$
gives rise to an $ O(V)$-torsor via $W \mapsto Isom(V,W)$. Note
that this is an $O(V)$-torsor because locally in the isovariant
\'etale topology, $W \cong V$,
so that locally $Isom(V,W) \cong Isom(V,V) \cong 
O(V)$. Because Hermitian vector bundles are isovariant \'etale locally
determined by rank, the same proof as the ordinary vector bundle case shows
that the category of $\mc O(V)$ torsors is equivalent to the category
of Hermitian vector bundles. Because over an affine scheme, every
Hermitian vector bundle is a summand of a hyperbolic module, it
follows that $S_{|V|}(V\perp \mbb H^\infty)(R)$ is equivalent to the
category of isovariant \'etale $\mc O(V)$ torsors. 

Let $\mc F : \Sm{S,qp}^{C_2} \rightarrow Gpd$ be the sheaf which assigns
to $f : X \rightarrow S$ the groupoid of $ O(f^*V)$-torsors. The
construction $W \mapsto Isom(f^*V,W)$ described above defines a
functor $S_{|V|}(V\perp \mbb H^\infty) \rightarrow \mc F$ which is an
equivalence when evaluated at affine $C_2$-schemes. It follows that
there's a sequence
\[
B S_{|V|}(V\perp \mbb H^\infty) \rightarrow B\mc F \rightarrow B_{isoEt}O(V)
\]
where the first map is a weak equivalence of simplicial sets when
evaluated at affine $C_2$-schemes, and by \cite{Jar01} Theorem 6, the second map is a
weak equivalence of simplicial sets when evaluated at any
$C_2$-scheme. 
\end{proof}

\begin{definition}
Following \cite{SchTri}, let
\[
\mc S_\bullet = \colim_{V \subset \mbb H^\infty_S} \mc S_{|V|}(V \perp
\mbb H^\infty)
\]
where similarly to the definition of $\RGr$, for $V \subset V'$ the
functor 
\[
\mc S_{|V|}(V \perp
\mbb H^\infty) \rightarrow \mc S_{|V'|}(V' \perp
\mbb H^\infty)
\]
is defined on objects by $E \mapsto E \perp V'-V$ and on morphisms by
$f \mapsto f \perp 1_{V'-V}$. 
\end{definition}

\begin{definition}
Define the infinite orthogonal group
\[
O(\mbb H^\infty_S) = \colim_{W \subseteq \mbb H^\infty_S} O(W).
\]
where the colimit is over non-degenerate subbundles of $\mbb H^\infty$. 
If $V$ is a Hermitian vector bundle, define  
\[
  O(V\perp \mbb H^\infty_S) = \colim_{W \subseteq V \perp \mbb H^\infty_S} O(W)
\]
where $W$ is a non-degenerate subbundle of $V \perp \mbb H^\infty$. 
\end{definition}

\begin{definition}
Let $R$ be a commutative ring. Define $\Delta R$ to be the simplicial ring with involution $[n] \mapsto
R[x_0,\dots,x_n]/(\sum x_i -1)$, where the involution is inherited
from the involution on $R$.
\end{definition}

\begin{lemma}(\cite{SchTri})\label{lem:incHE}
Let $V$ be a non-degenerate Hermitian vector bundle over a commutative
ring with involution $(R,\sigma)$ such that $\frac{1}{2} \in R$. Then
the inclusion $\mbb H^\infty \subset V \perp \mbb H^\infty$ induces a
homotopy equivalence of simplicial groups
\[
O(\mbb H^\infty)(\Delta R) \rightarrow O(V \perp \mbb
H^\infty)(\Delta R) \qquad A \mapsto 1_V \perp A.
\]
\end{lemma}

\begin{proof}
First, assume that $V = \mbb H$. Consider the map $j : O(\mbb H^n)
\rightarrow O(\mbb H^{2n+2})$ sending $A$ to $1_{\mbb H} \perp A \perp
1_{\mbb H^{n+1}}$. We claim that this is na\"ively $\mbb A^1$ homotopic
to the inclusion $i : O(\mbb H^n) \rightarrow O(\mbb H^{2n+2})$, $i(A)
= A \perp 1_{\mbb H^{n+2}}$ which defines the colimit $O(\mbb
H^\infty)$. Let $g = \begin{pmatrix}
0 & I_{2n} & 0 \\
I_2 & 0 & 0 \\
0 & 0 & I_{2n+2}
\end{pmatrix}$ where $I_n$ denotes an $n \times n$ identity
matrix. Then $i = gjg^{-1} = gjg^t$. Because $g$ corresponds to an
even permutation matrix, it can be written as a product of elementary
matrices, each of which is na\"ively $\mbb A^1$ homotopic to the
identity. It follow that $g$ is na\"ively $\mbb A^1$ homotopic to the
identity, and hence the induced maps $i,j : O(\mbb H^n)(\Delta R)
\rightarrow O(\mbb H^{2n+2})(\Delta R)$ are simplicially homotopic
via a base-point preserving homotopy. It follows that $i,j$ induce the
same map on homotopy groups, so that $j_* = i_*: \pi_kO(\mbb
H^\infty)(\Delta R) = \colim_n \pi_k O(\mbb H^n)(\Delta R)
\rightarrow \pi_k O(\mbb H^\infty)(\Delta R)$ is the colimit of a map
corresponding to a cofinal inclusion of diagrams, and hence is an
isomorphism on all simplicial homotopy groups. Because simplicial
groups are Kan complexes, it follows that $j$ is a homotopy
equivalence, and the claim is proved when $V = \mbb H$.

Now a trivial induction shows that the lemma holds when $V = \mbb
H^n$. In general, choose an embedding $V \subseteq \mbb H^n$, and
consider the sequence of maps
\[
O(\mbb H^\infty)(\Delta R) \rightarrow O(V \perp \mbb
H^\infty)(\Delta R) \rightarrow O(\mbb H^n \perp \mbb H^\infty)(\Delta R)
\]
\[
\rightarrow O(\mbb H^n \perp V \perp \mbb H^\infty)(\Delta R).
\]

The composites $O(\mbb H^\infty)(\Delta R) \rightarrow O(\mbb H^n
\perp \mbb H^\infty)$ and $O(V \perp \mbb
H^\infty)(\Delta R) \rightarrow O(\mbb H^n \perp V \perp \mbb
H^\infty)(\Delta R)$ are weak equivalences, so by 2 out of 6 the
first map is a weak equivalence. Because it is a map of simplicial
groups it is a homotopy equivalence. 
 
\end{proof}

For non-degenerate Hermitian vector bundles $(V,\phi_V), (W,\phi_W)$ and a commutative
$R$-algebra with involution $(A,\sigma)$, let
\[
\St(V,W)(A)
\]
be the set of $A$-linear isometric embeddings $f : V_A \rightarrow
W_A$. Given a map $A \rightarrow B$ of commutative $R$-algebras with involution, tensoring over $R$ with $B$ makes
$\St(V,W)(-)$ a presheaf on commutative $R$-algebras with involution. There's a
transitive left action of $O(V \perp \mbb H^\infty)$ on $\St(V,V \perp \mbb
H^\infty)$ given by $(f,g) \mapsto f \circ g$. Let $i_V$ denote the
isometric embedding $V \hookrightarrow V \perp \mbb H^\infty : v
\mapsto (v,0)$. The stabilizer of $i_V$ is the subgroup $O(\mbb
H^\infty) \subset O(V \perp \mbb H^\infty)$ where the inclusion map is
$A \mapsto 1_V \perp A$.

It follows that there's an isomorphism of presheaves of sets 
\[
O(\mbb H^\infty)\backslash O(V \perp \mbb  H^\infty) \cong
\St(V,V\perp \mbb H^\infty) \qquad f \mapsto f \circ i_V.
\]
Now Lemma \ref{lem:incHE} shows that the map $O(\mbb H^\infty)({\Delta
  R}) \rightarrow O(V \perp \mbb H^\infty)(\Delta R)$ is an
equivariant map which is a non-equivariant homotopy equivalence. The
simplicial group $O(\mbb H^\infty)(\Delta R)$ acts freely on both the
domain and codomain, so that the quotients $O(\mbb H^\infty)({\Delta
  R})\backslash O(V \perp \mbb H^\infty)(\Delta R)$ and $O(\mbb H^\infty)({\Delta
  R})\backslash O(\mbb H^\infty)(\Delta R)$ are homotopy equivalent.

Together with the isomorphism of simplicial sets
\[
O(\mbb H^\infty)({\Delta
  R})\backslash O(V \perp \mbb H^\infty)(\Delta R) \cong \St(V,V
\perp \mbb H^\infty)(\Delta R)
\]
it follows that $\St(V,V\perp \mbb H^\infty)(\Delta R)$ is
contractible for a commutative ring $(R,\sigma)$ with involution and
$\frac{1}{2} \in R$. Morever, this simplicial set is fibrant because
$G/H$ is fibrant for a simplicial group $G$ and subgroup $H$. We have
thus proved:

\begin{lemma}\label{lem:St_contrac}
Let $R$ be a commutative ring with $\frac{1}{2} \in R$. Then 
\[
\St(V,V\perp \mbb H^\infty)(\Delta R)
\]
is a contractible Kan set.
\end{lemma}

Now we move to identifying $\RGr_V$ as a quotient of a contractible
space by a free group action. Let $V$ be a non-degenerate Hermitian
vector bundle over a ring $R$ with involution. Then the group $O(V)$
acts on the right on $\St(V,U)$ by precomposition. The map $\St(V,U)
\rightarrow \RGr_V(U) : f \mapsto \im(f)$ factors through the quotient
$\St(V,U)/O(V)$. The map is clearly surjective, and hence furnishes an
isomorphism of sets
\[
\St(V,U)/O(V) \cong \RGr_V(U) \qquad f \mapsto \im(f).
\]
In particular, there's an isomorphism of presheaves of sets 
\begin{equation}\label{eq:St_RGr_inf}
  \St(V,V \perp \mbb
  H^\infty)/O(V) \cong \RGr_V(V \perp \mbb H^\infty) 
\end{equation}

Now, let $V$ be a non-degenerate Hermitian vector bundle over a ring with
involution $R$ and let $U$ be a possibly degenerate Hermitian form
over $R$. Define $\mc E_V(U)$ to be the category whose objects are
$R$-linear maps $V \rightarrow U$ of Hermitian forms (aka isometric embeddings), and whose
morphisms from two objects $a : V \rightarrow U$ and $b : V
\rightarrow U$ are maps $c : \im(a) \rightarrow \im(b)$ making the
diagram
\[
\xymatrix{V \ar[r]^a\ar[dr]_b & \im(a) \ar[d]^c\\ & \im(b)}
\]
commute.

There's a natural right action of $O(V)$ on $\mc E_V(U)$ which on
objects sends
\[
\mc E_V(U) \times O(V) \rightarrow \mc E_V(U) : (a,g) \mapsto ag
\]
and which on morphisms is the trivial action.
Then clearly there's an isomorphism
\[
\mc E_V(U)/O(V) \cong \mc S_V(U) \qquad a \mapsto \im(a).
\] 

\begin{lemma}
The category $\mc E_V(V \perp \mbb H^\infty)$ is contractible.
\end{lemma}

\begin{proof}
The category is nonempty and every object is initial.
\end{proof}

The map of simplicial sets
\[
\St(V,V\perp \mbb H^\infty)(\Delta R) \rightarrow \mc E_V(V \perp \mbb
H^\infty)(\Delta R)
\]
is $O(V)(\Delta R)$ equivariant and a weak equivalence after
forgetting the action. Furthermore, $O(V)(\Delta R)$ acts freely on
both sides, so that the induced map on quotients 
\begin{equation}\label{eq:RGr_Sv}
\RGr_V(V \perp \mbb
H^\infty)({\Delta R}) \rightarrow \mc S_V(V \perp \mbb
H^\infty)(\Delta R)
\end{equation}
 is also a weak equivalence. As an aside, the inclusion $BO(V) \subset B\mc S_V(V \perp \mbb
H^\infty)$ is a weak equivalence since $\mc S_V(V \perp \mbb
H^\infty)$ is a connected groupoid. 

We now show that there's a motivic equivalence $\RGr_\bullet
\rightarrow \colim_n B_{isoEt}O(\mbb H^n)$ over possibly non-regular Noetherian base rings. 

Let  $X \rightarrow S$ be an affine
$C_2$-scheme over $S$, and let $W$ be a non-degenerate Hermitian
vector bundle over $X$. Given an isovariant \'etale $O(W)$
torsor $\pi : T \rightarrow X$, and an isovariant \'etale torsor $U$, let $U_\pi$ denote
the twisted sheaf $(U \times T)/O(W)$. 

Our goal is to appy Lemma 2.1 from \cite{cdhdesc}, which we restate below:

\begin{lemma}(Hoyois)
Let $\Gamma$ be an isovariant \'etale sheaf of groups on $\Sm{S,qp}^{C_2}$
acting on an isovariant \'etale sheaf $U$. Suppose that, for every $X
\in \Sm{S,qp}^{C_2}$ and every isovariant \'etale torsor $\pi: T
\rightarrow X$ under $\Gamma$, $U_\pi \rightarrow X$ is a motivic
equivalence on $\Sm{S,qp}^{C_2}$. Then the map
\[
L_{isoEt}(U/\Gamma) \rightarrow B_{isoEt}\Gamma
\]
induced by $U \rightarrow \ast$ is a motivic equivalence on $\Sm{S,qp}^{C_2}$.
\end{lemma}

Given an $O(W)$-torsor $\pi : T \rightarrow X$, we want to check that $\St(W,\mbb H^\infty)_\pi  =(\St(W,\mbb H^\infty) \times T)/O(W) \rightarrow
X$ is a motivic equivalence on $\Sm{S,qp}^{C_2}$. Leting $V = W_\pi$,
this is equivalent to checking
that $\St(V,\mbb H^\infty_X)$ is motivically contractible
over $\Sm{X}^{C_2}$. To wit, because $X$ is affine there's an
embedding $V \hookrightarrow \mbb H^m$, and we have $V \perp \mbb
H^\infty_X \cong \mbb H^\infty_X$ since $\mbb H^\infty_X = \colim_{W
  \subset \mbb H^\infty_X} W$. It follows that $\St(V,\mbb H^\infty_X)
\cong \St(V,V \perp \mbb H^\infty_X)$, and Lemma
\ref{lem:St_contrac} (which didn't assume regularity of the base) shows
that $\St(V,V \perp \mbb H^\infty_X)$ is motivically contractible over
$\Sm{X}^{C_2}$.

It's a direct consequence of the above lemma that 
\[
L_{isoEt}(\St(W,W\perp \mbb H^\infty)/O(W))
\rightarrow B_{isoEt}O(W)
\]
is a motivic equivalence. However, we've already shown \eqref{eq:St_RGr_inf} that
\[
\St(W,W\perp \mbb H^\infty)/O(W) \cong \RGr_W(W \perp \mbb H^\infty),
\]
so that 
\[
L_{isoEt}\St(W,W\perp \mbb H^\infty)/O(W) \cong L_{isoEt}\RGr_W(W \perp
\mbb H^\infty) \cong \RGr_{|W|}(W \perp \mbb H^\infty)
\]
which after taking colimits gives the desired result.

\begin{theorem}\label{thm:RGr_GW_nonreg}
Let $S$ be a Noetherian scheme of finite Krull dimension with
$\frac{1}{2} \in S$. There are
equivalences of motivic spaces on $\Sm{S,qp}^{C_2}$
\[
\Z \times \RGr_\bullet \xrightarrow{\sim} \Z \times \colim_n
B_{isoEt}O(\mbb H^n)
\]
\end{theorem}
\section{Periodicity in the Hermitian $K$-Theory of Rings with Involution}\label{chap:Einf}

\subsection{A Projective Bundle Formula for $\mbb P^\sigma$}

Let $\mbb P^\sigma$ denote $\mbb P^1$ with involution $\sigma$ defined
by $[x:y] \mapsto
[y:x]$. When necessary we'll point it at the point $[1:1]$. Throughout this section, we'll fix the notation $\mc O = \mc
O_{\mbb P^1}$.

Consider the square of $\mc O$-modules
\begin{equation}\label{diag:prebott}
\xymatrix{\mc O(-1)\ar[r]^{\frac{T+S}{2}}\ar[d]^{\frac{T-S}{2}} & \mc O \ar[d]^{\frac{T-S}{2}} \\ \iHom(\sigma_*\mc
  O,\mc O) \ar[r]^{\frac{T+S}{2}} & \iHom(\sigma_* \mc O(-1),\mc O)}
\end{equation}

where the map $\frac{T-S}{2} : \mc O(-1) \rightarrow \mc
O$ is induced via the tensor-hom adjunction by the composition 
\begin{equation}\label{eq:TminSone}
  \begin{split}
\mc O(-1) \otimes \left\{\frac{T-S}{2}\right\} \otimes \sigma_* \mc O
&\xrightarrow{id \otimes i \otimes id}
\mc O(-1) \otimes \mc O(1) \otimes \sigma_*\mc O  
\\ &\xrightarrow{id \otimes id \otimes (\sigma^\#)^{-1}} \mc O(-1) \otimes
\mc O(1) \otimes \mc O  \xrightarrow{\mu \otimes id} \mc O \otimes \mc
O \xrightarrow{\mu} \mc O
  \end{split}
\end{equation}

and the map  $\frac{T-S}{2} : \mc O \rightarrow \sigma_*\mc
O(-1)$ is induced via the tensor-hom adjunction by the composition 
\begin{equation}\label{eq:TminS}
  \begin{split}
\mc O \otimes \left\{\frac{T-S}{2}\right\} \otimes \sigma_* \mc O(-1)
&\xrightarrow{\sigma^\# \otimes  \sigma^\# \circ i \otimes id}
\sigma_*\mc O \otimes \sigma_*\mc O(1) \otimes \sigma_*\mc O(-1)  
\\ &\xrightarrow{ id \otimes \sigma_*(\mu)} \sigma_*\mc O \otimes \sigma_*\mc
O \xrightarrow{\sigma_*\mu} \sigma_*\mc O
\xrightarrow{(\sigma^\#)^{-1}} \mc O
  \end{split}
\end{equation}
where $\mu$ denotes multiplication. We're abusing notation
in the map \eqref{eq:TminS} and using $\sigma^\#$ to denote both the maps $\mc O \rightarrow
\sigma_* \mc O$ and $\mc O(1) \rightarrow
\sigma_* \mc O(1)$ induced by the graded automorphism of $k[S,T]$
given by $f(S,T) \mapsto f(T,S)$. The image of $1 \in \mc O$ under the
adjoint of the map \eqref{eq:TminS}
 yields the element $\frac{S-T}{2}$ as a global section of
$\sigma_*\mc O(1)$. The map $\frac{T+S}{2} :
\iHom(\sigma_* \mc O, \mc O) \rightarrow \iHom(\sigma_* \mc O(-1),\mc
O)$ is induced by precomposition with $\sigma_*(\frac{T+S}{2})$, which is just
multiplication by the global section $\frac{T+S}{2}$. Equation
\eqref{eq:TminSone} can be understood similarly.

We
claim that diagram \ref{diag:prebott} commutes. Fix an open $U
\subseteq \mbb P^\sigma$ which need not be invariant, and open $V
\subseteq U$. Going down then right yields the
composite map
\[
u \mapsto (v \mapsto \frac{T-S}{2} \cdot u \cdot
(\sigma^\#)^{-1}\left(\frac{T+S}{2}\cdot v\right)).
\]

Going right first then down yields the composite
\[
u \mapsto (v \mapsto (\sigma^\#)^{-1}(\sigma^\#(\frac{T+S}{2} \cdot u
)\cdot \frac{S-T}{2}\cdot v))
\]
These are equal since $\frac{T+S}{2}$ is an invariant global
section. Note that the diagram \ref{diag:prebott} is a map in\\
$\Fun([1],\Vect(\mbb P^\sigma))$ from 
\[
\mc O(-1) \xrightarrow{\frac{T+S}{2}} \mc O
\]
to its dual,
\[
\iHom(\sigma_*\mc
  O,\mc O) \xrightarrow{\frac{T+S}{2}}  \iHom(\sigma_* \mc O(-1),\mc O).
\]
Thus this diagram defines a (not necessarily non-degenerate) form,
which we denote by $\phi$. 

In order to show that this $\phi$ is symplectic, we have to check that
$\phi^* \circ (-\can) = \phi$. To spell this out in detail, the dual
and double dual are functors. Applying these two functors, we get the
two objects
\[
\xymatrix{O^* \ar[r]^{\frac{T+S}{2}^\ast} & O(-1)^*}
\]
and
\[
\xymatrix{O(-1)^{\ast\ast} \ar[r]^{\frac{T+S}{2}^{\ast\ast}} & O^{\ast\ast}}
\]
in  $\Fun([1],\Ch^b\Vect(\mbb P^\sigma))$.

Because $\can$ is a natural transformation $id \rightarrow \ast\ast$,
there's a commutative diagram
\[
\xymatrix{O(-1) \ar[d]^{\can} \ar[r]^{\frac{T+S}{2}} & O \ar[d]^{\can}
  \\ O(-1)^{\ast\ast} \ar[r]^{\frac{T+S}{2}^{\ast\ast}}
  \ar[d]^{\frac{T-S}{2}^\ast} & \ar[d]^{\frac{T-S}{2}^\ast}  O^{\ast\ast}
\\O^{\ast} \ar[r]^{\frac{T+S}{2}^{\ast}} & O(-1)^{\ast}
}
\]

The goal is to show that the vertical maps in the large rectangle are
the negative of the vertical maps in diagram \ref{diag:prebott}.
Tracing
through the definitions, we see that $\can$ is the map which sends $u
\in \mc O(-1)(U)$ to the natural transformation
\[
\gamma \mapsto (\sigma^\#)^{-1}(\gamma(u|_V)),
\]
and $\phi^* \circ \can (u)$ is the natural transformation
\[
v \mapsto (\sigma^\#)^{-1}\left(\frac{T-S}{2} \cdot v \cdot (\sigma^\#)^{-1}(u)\right)
\]
which is the same thing as
\[
v \mapsto \left(-\frac{T-S}{2} \cdot (\sigma^\#)^{-1} (v) \cdot u\right).
\]

On the other hand, $\frac{T-S}{2} :
\mc O(-1) \rightarrow \mc O^*$ is the map
\[
u \mapsto (v \mapsto \frac{T-S}{2} \cdot u \cdot (\sigma^\#)^{-1}(v))
\]
which is by what we calculated above equal to $-(\phi^* \circ \can) =
\phi^* \circ (-\can)$.

Now just as in \cite{Schder}, taking the mapping cone of $\phi$ via
the functor

\[
\Cone : \Fun([1],\Ch^b\Vect(\mbb P^\sigma))^{[0]} \rightarrow
\left(\Ch^b\Vect(\mathbb P^\sigma)\right)^{[1]}
\]
yields a symplectic form $\beta^\sigma = \Cone(\phi)$. 

We claim that there's an exact sequence
\[
\mc O(-1) \xrightarrow{\begin{pmatrix}\frac{T+S}{2} \\ \frac{T-S}{2} \end{pmatrix}} \mc O \oplus \mc O^* \xrightarrow{\begin{pmatrix}\frac{T+S}{2} & -\frac{T-S}{2} \end{pmatrix}} \mc O(-1)^*
\]
where the maps are the maps in diagram \ref{diag:prebott}. The fact
that the composite is zero follows from commutativity of that
\ref{diag:prebott}. To show that the kernel equals the image, note
that any permutation of $(\frac{T+S}{2},\frac{S-T}{2})$ is a regular sequence on
$k[S,T]$. Thus if  $\frac{T+S}{2}x + \frac{S-T}{2}y = 0$, reducing mod
$\frac{T+S}{2}$ we see that $y \in (\frac{T+S}{2})$ and reducing mod
$\frac{S-T}{2}$ we see that $x \in (\frac{S-T}{2})$. It follows that the square defining $\phi$ is a pushout, and hence the induced map on
mapping cones is a quasi isomorphism. Hence $\beta^\sigma$
is a well-defined, non-degenerate symplectic form in
$\left(\Ch^b\Vect(\mathbb P^\sigma)\right)^{[1]}$.

\begin{theorem} \label{thm:proj_inv}
Let $X$ be a scheme with involution, an ample family of line
bundles, and $\frac{1}{2} \in X$, and denote by $p : \mbb P^\sigma
\rightarrow X$ the structure map of the equivariant projective line
over $X$, with action $[x:y] \mapsto [y:x]$. Then for all $n \in \mbb
Z$, the following are natural stable equivalences of (bi-) spectra
\begin{align*}
GW^{[n]}(X) \oplus & GW^{[n-1]}(X,-\can) \xrightarrow{\sim}
GW^{[n]}(\mbb P^\sigma_X)\\
 \mbb GW^{[n]}(X) \oplus & \mbb GW^{[n-1]}(X,-\can) \xrightarrow{\sim}
\mbb GW^{[n]}(\mbb P^\sigma_X)\\
&(x,y) \mapsto p^*(x) + \beta^\sigma \cup p^*(y).
\end{align*}
\end{theorem}

\begin{proof}
The proof of Theorem 9.10 in \cite{Schder} can be easily adapted. Note
that our Bott element $\beta^\sigma$ is  a linear change of
coordinates from the standard Bott element on $\mbb P^1$. Keeping in
mind that the involution only affects the duality and not the
underlying derived category with weak equivalences, it's still true
that $\beta^\sigma \otimes : \mc T\mathrm{sPerf}(X) \rightarrow \mc
T\mathrm{sPerf}(\mbb P^1_X)/p^*\mc T\mathrm{sPerf}(X)$ is an
equivalence of triangulated categories. As in
\textit{loc. cit.}, if we denote by $w$ the set of morphisms in
$\sPerf(\mbb P^1_X)$ which are isomorphisms in $\mc T\sPerf(\mbb
P^1_X)/p^*\mc T\sPerf(X)$, we get a sequence
\[
\xymatrix{(\sPerf(X),\mathrm{quis}) \ar[r]^{p^*} & (\sPerf(\mbb
  P^1_X),\mathrm{quis}) \ar[r] & (\sPerf(\mbb P^1_X),w)}
\]
which is a Morita exact sequence of categories with duality. That is,
the maps are maps of categories with duality, and the underlying
sequence of categories is Morita exact. It follows that this sequence
induces a homotopy fibration of $GW^{[n]}$ and $\mbb GW^{[n]}$
spectra. As remarked above, these fibration sequences split via the
exact dg form functors
\[
\xymatrix{(\sPerf(X),\mathrm{quis}) \ar[r]^{\beta^\sigma \otimes} & (\sPerf(\mbb
  P^1_X),\mathrm{quis}) \ar[r] & (\sPerf(\mbb P^1_X),w)}
\]
so that the composite induces an equivalence of triangulated categories.
Finally, using that $GW$ and $\mbb GW$ are invariant under derived
equivalences \cite[Theorem 6.5]{Schder} \cite[Theorem 8.9]{Schder}, we conclude the theorem.
\end{proof}

Considering $GW$ as a presheaf of spectra on $\Sch{S,qp}^{C_2}$ it follows from
Theorem \ref{thm:proj_inv} that $GW^{[n]}(\mbb P^\sigma,[1:1]) \cong
GW^{[n-1]}(X,-\can) \cong GW^{[n+1]}(X)$, recovering one of the results of
  \cite{Xie2018ATM}. Hence
\[
\iHom(\Sigma^\infty (\mbb P^\sigma,[1:1]),GW^{[n]}) \cong
GW^{[n+1]}
\]
as presheaves of spectra on $\Sch{S,qp}^{C_2}$. In particular, by the
projective bundle formula from \cite{Schder} and the usual cofiber
sequence 
\[
([1:1] \times \mbb P^\sigma) \vee (\mbb P^1 \times [1:1]) \rightarrow
\mbb P^\sigma \times \mbb P^1 \rightarrow \mbb P^\sigma \wedge \mbb P^1
\]
we obtain the periodicity isomorphism  
\[
\iHom((\mbb P^1,[1:1]) \wedge (\mbb P^\sigma,[1:1]),GW^{[n]}) \cong
GW^{[n]}
\]
induced by the map 
\begin{align*}
GW^{[n]}(X) &\rightarrow GW^{[n+1]}(\mbb P^1_X) \rightarrow
              GW^{[n]}(\mbb P^\sigma_{\mbb P^1_X})\\
& x \mapsto \beta \cup p^*(x) \mapsto \beta^\sigma \cup q^*(\mc O_X[-1]
  \otimes \beta \cup p^*(x))
\end{align*}
where $p$ is the projection $\mbb P^1_X \rightarrow X$, and $q$ is the
projection $\mbb P^\sigma_{\mbb P^1_X} \rightarrow \mbb P^1_X$. The
analogous statements hold for the presheaf of spectra $\mbb GW$.

As notation for later, let $\beta^{1+\sigma}$ denote the induced map
\begin{equation}\label{eq:bott_element}
\beta^{1+\sigma} : (\mbb P^1,[1:1]) \wedge (\mbb P^\sigma,[1:1])
\rightarrow GW.
\end{equation}

\begin{lemma}\label{lem:bott_zero}
The Bott element $\beta^\sigma$ restricts to zero in $C_2 \times \mbb
A^\sigma = \mbb P^{\sigma} - [1:0] \coprod \mbb P^\sigma - [0:1]$. 
\end{lemma}

\begin{proof}
As in \cite{Schder}, because the Bott element is natural it suffices
to prove that the bott element $\beta^\sigma$ in $\mbb
P^\sigma_{\Z[\frac{1}{2}]}$ restricts to zero. From the definition of
the Bott element, it's clear that it's supported on $[1:-1]$. There's a commutative diagram
\[
\xymatrix{GW^{[n]}(C_2 \times \mbb A^\sigma \text{ on } [1:-1] \coprod
[1:-1]) \ar[r]^-k & GW^{[n]}(C_2 \times \mbb A^\sigma) \ar[d]^f \ar[dl]_g \\
GW^{[n]}(C_2 \times \mbb A^\sigma - [1:-1]) \ar[r]^h &
GW^{[n]}(C_2 \times \Spec(\Z[\frac{1}{2}])) 
} 
\]
where $f$ and $h$ are induced by inclusion of the point
$[1:1]$. Because $\Z[\frac{1}{2}]$ is regular and $C_2 \times
\mbb A^\sigma$ is equivariantly isomorphic to $C_2 \times \mbb A^1$, 
\cite[Theorem 7.5]{Xie2018ATM} shows that $f$ is an isomorphism,
hence $g$ is an injection. By localization
\cite[Theorem 6.6]{Schder}, the maps $k$ and $g$ compose to form an exact sequence,
 and it follows that $k$ is the zero
map. 
\end{proof}

\subsection{The Periodization of $GW$}

The idea behind the Bass construction in algebraic $K$-theory is that
as a consequence of satisfying localization, there is a
Bass exact sequence ending in
\[
\cdots \rightarrow K_n(\mbb G_m) \xrightarrow{\partial} K_{n-1}(X) \rightarrow 0
\]
for all $n$. This comes from applying $K$-theory to the pushout square manifesting
the usual cover of $\mbb P^1$ together with the projective bundle formula. The map $\partial$ is split by $x
\mapsto [T] \cup p^*(x)$ where $p$ is the projection to the base scheme $p : \mbb G_m
\rightarrow X$. It follows that if $K$ exhibits an exact Bass sequence
in all degrees $n$, then $K_{n-1}(X)$ can be identified with
the image of $\partial([T]) \cup -$ (i.e. this map is an automorphism
of $K_{n-1}(X)$. In fact, $\partial([T]) \cup -$ is the idempotent endomorphism
$(0,1)$ of $K_{0}(\mbb P^1) \cong K_0(X) \oplus K_0(X)$).
 The Bass construction can be thought of as defining
$K^B_n(X)$ so that there's an exact sequence $K^B_n(\mbb A^1) \oplus
K^B_n(\mbb A^1) \rightarrow K^B_n(\mbb G_m) \rightarrow K^B_{n-1}(X)$,
then identifying $K^B_{n-1}(X)$ with $(0,1) \cdot
K^B_{n-1}(\mbb P^1)$. In other words, it can be constructed as the colimit
\[
K^B = \colim(K \rightarrow \iHom(\mbb A^1 \coprod_{\mbb G_m} \mbb
A^1,K) \rightarrow \dots)
\] 
where the pushouts are taken in presheaves and the maps are induced by applying $\iHom(-,K)$ in the category of
$K$-modules to the composite
\[
\mbb A^1 \coprod_{\mbb G_m} \mbb A^1 \rightarrow \Sigma \mbb G_m
\xrightarrow{T} K.
\]

Here, loosely speaking, the first map in the composite represents the
boundary in the long Bass exact sequence $\partial$
while the second represents $[T]$, so that in the category of $K$
modules this map represents cup product with $\partial([T])$. 

We'll spell out an example a bit more explicitly to give a flavor for the
constructions to come. Let $W = \mbb A^1 \coprod_{\mbb G_m} \mbb A^1$,
where we emphasize again that the pushout is in the category of
presheaves. Because this is a (homotopy) pushout in the category of
presheaves, applying $\iHom(-,K)$ gives us a homotopy pullback square,
and hence a Mayer-Vietoris long exact sequence. In particular, it gives us a map
of presheaves of spectra (which can be promoted to a map of $K$-modules) $\Omega  K(\mbb G_m) \rightarrow K(W)$, where
we abuse notation and write $K(W)$ for the internal hom of $W$ into
$K$. Because $K^B$ satisfies Nisnevich descent, and $K_i(-) =
K^B_i(-)$ for $i \geq 0$, it follows by the 5-lemma that $K_0(W) \cong
K_0(\mbb P^1) \cong K_0(X) \oplus K_0(X)$, and that the element
$\partial([T]) \cup -$ represents projection onto the second factor as an
endomorphism of $K_0(W)$. 

Now,
we want to explain why $\partial([T]) \cup K_{-1}(W) \cong
K^B_{-1}(X)$. We'll use the fact that $K^B(W) \cong K^B(\mbb P^1)$ and that
$K^B_{-1}(X) = \partial([T]) \cup K^B_{-1}(\mbb P^1)
= \partial(K^B_0(\mbb G_m))$. 

To begin, because $\partial([T])$ is zero in
$K_0(\mbb A^1)$, the image of $\partial([T]) \cup K_{-1}(W)$ in
$K_{-1}(\mbb A^1) \oplus K_{-1}(\mbb A^1)$ is zero. By exactness, it follows that
$\partial([T]) \cup K_{-1}(W)  \subseteq  \partial K_0(\mbb G_m)$. 

There's a map $\phi : K_{-1}(W) \rightarrow K_{-1}^B(W) \cong
K_{-1}^B(\mbb P^1)$ and a commutative diagram
\[
\xymatrix{K_0(\mbb A^1) \oplus K_0(\mbb A^1) \ar[r] \ar[d] & K_0(\mbb
  G_m) \ar[r]^\partial \ar[d] & K_{-1}(W) \ar[d]^\phi\\
K_0(\mbb A^1) \oplus K_0(\mbb A^1) \ar[r] & K_0(\mbb
  G_m) \ar[r]^\partial & K_{-1}^B(W)
}
\]
which shows that $\phi$ restricts to an isomorphism $\partial(K_0(\mbb
  G_m)) \cong \partial(K_0^B(\mbb
  G_m))$, and in particular that $\phi(\partial(K_0(\mbb
  G_m)) ) = \partial(K_0^B(\mbb
  G_m))$. Now 
\[
\phi(\partial([T]) \cup \partial(K_{0}(\mbb G_m)) = \partial([T]) \cup
  \phi(\partial K_{0}(\mbb G_m)) = \partial([T]) \cup \partial
  K_{0}^B(\mbb G_m) = \partial K_0^B(\mbb G_m)
\]
where we've crucially used that for Bass $K$-theory, $\partial([T])
\cup \partial K_0^B(\mbb G_m) = \partial([T])
\cup \partial K_0^B(\mbb P^1) = \partial K_0^B(\mbb G_m)$. 

But as remarked above, the fact that $\partial([T])$ is trivial in
$K_0(\mbb A^1)$ implies that $\partial([T]) \cup \partial(K_{0}(\mbb G_m))
\subseteq \partial(K_{0}(\mbb G_m))$, and we know that
$\phi|_{\partial K_0(\mbb G_m)}$ is an isomorphism. Since
$\phi|_{\partial([T]) \cup \partial(K_{0}(\mbb G_m))}$ is surjective
by the chain of equalities above,
it follows that $\partial
K_0(\mbb G_m) =\partial([T]) \cup K_{-1}(W)$. We've shown that
\[
\partial([T]) \cup K_{-1}(W) = \partial K_0(\mbb G_m) \cong \partial
K_0^B(\mbb G_m) \cong K_{-1}^B(X).
\]

If we take pointed versions of the above sequences by pointing all the
schemes in question at $[1:1]$ everything goes through as above with
the extra benefit that $\partial([T]) \cup K^B_{-1}(W,1) =
K^B_{-1}(W,1)$, and the map $K^B(X) \rightarrow K^B(W,1)$, $x \mapsto
p^*(x) \cup \partial([T])$ is an isomorphism by the projective bundle
formula. Now the map $p : (W,1) \rightarrow 1$ is split by inclusion
of the base point, and thus $p^* : K(W,1) \rightarrow K((W,1) \otimes
(W,1))$ is injective. Furthermore, $p^*(x \cup \partial([T]))
= \partial([T]) \cup p^*(x)$, so that the image of $K_{-1}(W,1)$ in
$K_{-1}((W,1) \otimes (W,1))$ under the map $x \mapsto \partial([T])
\cup p^*(x)$ is, by what we showed above, isomorphic
to $K^B_{-1}(X)$. This shows that
\[
\pi_{-1}K^B = \pi_{-1}\colim(K \rightarrow \iHom(\mbb A^1 \coprod_{\mbb G_m} \mbb
A^1,K) \rightarrow \cdots).
\] 

This argument is mostly formal given a few pieces of structural
information:
\begin{itemize}
\item A map $K \rightarrow K^B$ which respects cup
  products,
\item Nisnevich descent for $K^B$, and
\item A Bass exact sequence split by cup product with an element in
  $K_1(\mbb G^m)$.
\end{itemize}

The remainder of this section will show that these three pieces of
structure are present
for Grothendieck-Witt groups, which will allow us to repeat
essentially the same argument to give a construction of the localizing
$\mbb GW$ as a periodization of $GW$. When the base scheme is a
perfect field, a similar construction
of $GW$ as a periodic spectrum was given in \cite{HuKriz}.

First, equivariant Nisnevich
descent for $\mbb GW$ is a consequence of results from \cite{Schder}.

\begin{lemma}
$\mbb GW$ is Nisnevich excisive on the category of schemes with
an ample family of line bundles over $S$. 
\end{lemma}

\begin{proof}
Recall that the distinguished squares defining the equivariant
Nisnevich cd-structure are cartesian squares in $\Sch{S,qp}^G$ 

\[
\xymatrix{B \ar[r]\ar[d] & Y \ar[d]^p \\ A \ar[r]^j & X}
\]

where $j$ is an open immersion, $p$ is \'etale, and
$(Y-B)_{\mathrm{red}} \rightarrow (X-A)_{\mathrm{red}}$ is an
isomorphism. 

As in \cite[Theorem 9.6]{Schder}, a result of Thomason \cite[Theorem 2.6.3]{Thomason2007} tells us that the map $p$
induces a quasi-equivalence of dg categories
\[
p^* : \mathrm{sPerf}_Z(X) \rightarrow \mathrm{sPerf}_Z(Y).
\]
Because $\mbb GW$ is invariant under derived
equivalences \cite[Theorem 8.9]{Schder}, it follows that $p^*$ induces
an isomorphism on Grothendieck-Witt groups. Noting that $U$ and the closed
subset $Z = X-A$ are $G$-invariant, the localization
sequence \cite[Theorem
9.5]{Schder} generalizes to our setting and identifies $\mbb GW(\mathrm{sPerf}_Z(X))$ and $\mbb
GW(\mathrm{sPerf}_Z(Y))$ as the horizontal homotopy fibers. This allows us to conclude the result. 
\end{proof}

 Next, we
identify the analogues of the Bass sequence and the splittings
therein. From \cite[Theorem 9.13]{Schder}, we know that there's a Bass sequence
\[
\xymatrix{0 \ar[r] & \mbb GW^{[n]}_i(X) \ar[r] & \mbb GW^{[n]}_i(\mbb
  A^1_X) \oplus \mbb GW^{[n]}_i(\mbb A^1_X) \\ \ar[r] & \mbb
  GW^{[n]}_i(X[T,T^{-1}]) \ar[r] & \mbb GW^{[n-1]}_{i-1}(X) \ar[r] & 0}
\]
where the last non-trivial map is split by cup product with (the
pullback of) $[T]$ in $\mbb
GW_1^{[1]}(\Z[\frac{1}{2}][T,T^{-1}])$. This gives us a candidate map
$\mbb A^1 \coprod_{\mbb G_m} \mbb A^1 \rightarrow \Sigma \mbb G_m \xrightarrow{[T]} GW^{[1]}$. 

Now, we want to find a candidate map $\Sigma^\sigma \mbb G_m^\sigma
\rightarrow GW^{[-1]}$ so that we can eventually invert 
\[
\Sigma^\sigma
\mbb G_m^\sigma \otimes \Sigma \mbb G_m \rightarrow GW^{[-1]} \otimes
GW^{[1]} \rightarrow GW^{[0]}.
\]
Define $W_\sigma$ by the pushout square in
the category of presheaves
\[
\xymatrix{(C_2 \times \mbb G_m^\sigma)_+ \ar[r]\ar[d] & (C_2 \times \mbb
  A^\sigma)_+ \ar[d] \\ (\mbb G_m^\sigma)_+ \ar[r] & W_\sigma}
\]
There's an associated homotopy pushout square
\[
\xymatrix{(C_2 \times \mbb G_m^\sigma)_+/(C_2)_+ \ar[r]\ar[d] & (C_2 \times \mbb
  A^\sigma)_+/(C_2)_+ \ar[d] \\ (\mbb G_m^\sigma)_+/S^0 \ar[r] & W_\sigma/S^0}
\]
and taking the homotopy cofiber of the left vertical map yields
$S^\sigma \wedge \mbb G_m^\sigma$. It follows that the homotopy cofiber
of the right vertical map is equivalent to  $S^\sigma \wedge \mbb
G_m^\sigma$, and that there's a long exact sequence
\begin{equation}\label{eq:bass_inv}
\xymatrix{\cdots \ar[r] & GW_i^{[n]}(S^\sigma \wedge \mbb G_m^\sigma) \ar[r] &
  GW_i^{[n]}(W_\sigma/S^0) \\ \ar[r] & GW_i^{[n]}((C_2 \times \mbb
  A^\sigma)_+/(C_2)_+) \ar[r] & \cdots }
\end{equation}
Here if $\mbb A^\sigma_S \cong S$, then $(C_2 \times \mbb
  A^\sigma)_+/(C_2)_+ \cong (C_2)_+ \wedge \mbb A^\sigma$ is
  contractible and $W/S^0 \cong S^\sigma \wedge \mbb G_m^\sigma$. Working over the regular ring $\Z[\frac{1}{2}]$, $GW(W_\sigma/S^0)
\cong GW(\mbb P^\sigma/S^0)$, and
\[  
GW^{[n]}_i(W_\sigma/S^0) \cong GW^{[n]}_i(\mbb P^\sigma/S^0) \cong GW^{[n+1]}_i(S)
\]
by the projective bundle formula \ref{thm:proj_inv}. 

The maps in the sequence
  \eqref{eq:bass_inv} are maps of $GW^{[0]}_*$-modules, and the
  sequence is natural in the base scheme. The induced map
\[
GW^{[-1]}_0(S^\sigma \wedge \mbb G_m^\sigma) \rightarrow GW^{[0]}_0(\Z[\frac{1}{2}])
\]
is an isomorphism of $GW^{[0]}_0(\Z[\frac{1}{2}])$-modules, and hence
the inverse is uniquely determined by a lift of the element $\langle 1 \rangle
\in GW^{[0]}_0(\Z[\frac{1}{2}])$ to $GW^{[-1]}_0(S^\sigma \wedge \mbb
G_m^\sigma)$. We stress that this element $\langle 1 \rangle$ maps to
$\beta^\sigma \cup \mc O_{\Z[\frac{1}{2}]}[-1] \cup \langle 1 \rangle$ in $GW(\mbb P^\sigma)$,
and in particular it isn't the unit of multiplication in $GW(\mbb P^\sigma)$. We'll denote this element by $[T^\sigma]$ in analogy with
the non-equivariant case. 

Over an arbitrary base scheme $X$, we denote by $[T^\sigma]$ the pullback of
$[T^\sigma]$ to $GW^{[-1]}_0(S^\sigma \wedge \mbb G_m^\sigma
\times_{\Z[\frac{1}{2}]} X)$ using functoriality of $GW$. We summarize
in the definition below.

\begin{definition}\label{def:T_Tsigma}
Let  $[T]$ denote the class of the element $T$ in $\mbb
GW_1^{[1]}(\Z[\frac{1}{2}][T,T^{-1}])$. Let $\partial([T])$ denote the
image of $[T]$ under the connecting map in the Bass sequence
\[
\partial : GW_1^{[1]}(\Z[\frac{1}{2}][T,T^{-1}]) \rightarrow
GW_0^{[1]}(\mbb P^1_{\Z[\frac{1}{2}]}).
\]

Let $[T^\sigma]$ denote the lift of the element $\langle 1 \rangle
\in GW^{[0]}_0(\Z[\frac{1}{2}])$ to $GW^{[-1]}_0(S^\sigma \wedge \mbb
G_m^\sigma)$. Let $\partial([T^\sigma])$ denote the image of
$[T^\sigma]$ under the connecting map in the long exact sequence
\ref{eq:bass_inv}
\[
\partial : \xymatrix{GW_1^{[-1]}(S^\sigma \wedge \mbb G_m^\sigma) \ar[r] &
  GW_0^{[-1]}(W_\sigma/S^0)}.
\]

Over an arbitrary scheme $S$ with $\frac{1}{2} \in S$, let $[T]$ and
$[T^\sigma]$ denote the pullbacks $f^*([T])$, $f^*([T^\sigma])$ under
the unique map $f : S \rightarrow \mbb Z[\frac{1}{2}]$, and similarly
for $\partial([T])$ and $\partial([T^\sigma])$.
\end{definition}

Let $W = (\mbb A^1 \coprod_{\mbb G_m} \mbb A^1)_+$. Now (by taking the pointed version of everything) we have a candidate map
\begin{equation}\label{eq:gamma}
\gamma : W_\sigma/S^0 \otimes W/S^0 \rightarrow
S^\sigma \wedge \mbb G_m^\sigma \otimes S^1 \wedge \mbb G_m
\xrightarrow{[T^\sigma] \otimes [T]}
GW^{[-1]} \otimes GW^{[1]} \rightarrow GW
\end{equation}
to invert. 

Given a presentably symmetric monoidal $\infty$-category and a
morphism $\alpha : x \rightarrow \mathbf{1}$ to the monoidal unit,
define 
\[
Q_\alpha E = \colim(E \xrightarrow{\alpha} \iHom(x,E)
\xrightarrow{\alpha} \iHom(x^{\otimes 2}, E) \xrightarrow{\alpha}
\dots ).
\]
In general $Q_\alpha E$ is not the periodization of $E$ with respect
to $\alpha$, one obstruction being that the cyclic permutation of
$\alpha^3$ can fail to be homotopic to the identity. This matters
because checking periodicity requires permuting $\alpha \otimes id$ to $id \otimes \alpha$, and these
can fail to be homotopic. 

\begin{lemma}\label{lem:GWb_eq}
The canonical map $\mbb GW \rightarrow Q_\gamma \mbb GW$ is an
equivalence of (pre)sheaves of spectra on $\Sch{S,qp}^{C_2}$. 
\end{lemma}

\begin{proof}
We know by the projective bundle formulas
that 
\begin{align*}
\mbb GW(\mbb P^\sigma \times \mbb P^1 = \mbb
P^\sigma_{\mbb P^1}) &\cong \mbb GW(\mbb P^1) \oplus \mbb GW^{[1]}(\mbb
P^1)\\ &\cong \mbb GW(X) \oplus \mbb GW^{[-1]}(X) \oplus \mbb
GW^{[1]}(X) \oplus \mbb GW(X).
\end{align*}
We claim that under this isomorphism, cup product with
$\partial [T^\sigma]$ is projection onto $\mbb
GW^{[1]}(\mbb P^1)$ and cup product with $\partial [T]$ on
$GW^{[1]}(\mbb P^1)$ is projection
onto $\mbb GW(X)$. The latter statement is already known from
\cite[Theorem 9.10]{Schder}, so we show the former. It suffices to show that cup
product with $\partial [T^\sigma_{X}] \cup - : \mbb GW^{[n]}(X) \oplus
\mbb GW^{[n+1]}(X) \rightarrow \mbb GW(X) \oplus \mbb GW^{[n+1]}(X)$ is projection
onto the second factor. But this is precisely how $[T^\sigma]$ is
defined: it's a lift under $\partial$ of a generator of $\mbb
GW^{[1]}(X)$, so cup product with it is cup product with $\langle 1
\rangle$ on $GW^{[n+1]}(X)$
and it's necessarily zero on the other factor because it gives a
well-defined element on
the pointed $\mbb GW^{[-1]}(\mbb P^\sigma,[1:1])$.

Because $\mbb GW$ satisfies equivariant Nisnevich descent, $\mbb
GW(W/S^0) \cong \mbb GW(\mbb P^1,[1:1])$, and $\mbb GW(W_\sigma/S^0)
\cong \mbb GW(\mbb P^\sigma,[1:1])$. Now we're essentially done. The
maps in the colimit defining $Q_\gamma \mbb GW$ first identify $\mbb GW_i^{[n]}(X)$
with $\partial([T]) \cup \mbb GW_i^{[n]}(\mbb P^1_X,[1:1])$, then identify
$\mbb GW_i^{[n]}(\mbb P^1_X)$ with  $\partial([T^\sigma]) \cup
\mbb GW_i^{[n]}(\mbb P^\sigma_{\mbb P^1_X},[1:1])$. As we noted above,
the projective bundle formulas imply that the image of $\mbb GW_i^{[n]}(X)$ under these identifications is isomorphic
to $\mbb GW_i^{[n]}(X)$, and hence $Q_\gamma\mbb GW_i^{[n]}(X) \simeq \mbb
GW_i^{[n]}(X)$ as desired.

\end{proof}

\begin{lemma}\label{lem:pos_htpy}
The canonical map $Q_\gamma GW^{[m]} \rightarrow Q_\gamma \mbb GW^{[m]} \cong \mbb
GW^{[m]}$ induces isomorphisms $\pi_nQ_\gamma GW^{[m]} \cong \pi_n \mbb
GW^{[m]}$ for $n \geq 0$ and for all $m$.
\end{lemma}

\begin{proof}
This follows from two out of three and the proof of lemma \ref{lem:GWb_eq} since $\pi_nGW^{[m]} \cong \pi_n \mbb GW^{[m]}$
for $n \geq 0$ and for all $m$. 
\end{proof}

\begin{lemma}\label{lem:neg_htpy}
The canonical map $Q_\gamma GW^{[m]} \rightarrow Q_\gamma \mbb GW^{[m]} \cong \mbb
GW$ induces an isomorphism $\pi_n Q_\gamma GW^{[m]} \cong \pi_n \mbb
GW^{[m]}$ for $n <= 0$ and for all $m$.
\end{lemma}

\begin{proof}
Because homotopy groups
commute with filtered (homotopy) colimits of spectra
\[
\pi_nQ_\gamma GW^{[m]} = \colim(\pi_n GW \xrightarrow{\alpha} \pi_n\iHom(W/S^0
\otimes W_\sigma/S^0,GW) \xrightarrow{\alpha^{\otimes  2}} \cdots).
\]

Fix $[m]$ for now and denote by $F^i_n$ the image of the map of groups
\[
\gamma^* : GW_n^{[m]}((W/S^0 \otimes W_\sigma/S^0)^{\otimes i})
\rightarrow GW_n^{[m]}((W/S^0 \otimes W_\sigma/S^0)^{\otimes i+1})
\]
and note that $F^0_n \cong GW^{[m]}_n$. Denote by $FB^i_n$ the same
construction as above with $GW$ replaced by $\mbb GW$.

For $i \geq -n$, we claim that there are exact sequences
\begin{align*}
F^i_n(\mbb A^1/1 \otimes W_\sigma/S^0)  \oplus F^i_n(\mbb
  A^1/1 \otimes W_\sigma/S^0)  &\rightarrow  F^i_n(\mbb G_m/1 \otimes W_\sigma/S^0) \\  
  &\xrightarrow{\partial} 
  F^{i}_{n-1}(W/S^0 \otimes W_\sigma/S^0) 
\end{align*}
such that $\partial(F^i_n(\mbb
G_m/1 \otimes W_\sigma/S^0)) = \partial([T]) \cup \partial([T^\sigma])
\cup  F^{i}_{n-1}(W/S^0 \otimes W_\sigma/S^0)$. We prove this in conjunction with the
statement that, for each $n$,  $F^i_n \cong \mbb
GW^{[m]}_n$ for $i \geq -n$. The proof is induction in $i$, and we
must show that $\partial(F^i_n(\mbb
G_m/1 \otimes W_\sigma/S^0)) = \partial([T]) \cup  F^{i}_{n-1}(W/S^0
\otimes W_\sigma/S^0)$. For $n \geq 0$, the same argument that we
gave for $K$-theory together with lemma \ref{lem:pos_htpy} works. In
more detail, there's an exact sequence
\begin{align*}
GW^{[m]}_n(\mbb A^1/1 \otimes W_\sigma/S^0)  \oplus GW^{[m]}_n(\mbb
  A^1/1 \otimes W_\sigma/S^0) &\rightarrow  GW^{[m]}_n(\mbb G_m/1 \otimes W_\sigma/S^0)\\ 
  &\xrightarrow{\partial} 
  GW^{[m]}_{n-1}(W/S^0 \otimes W_\sigma/S^0)
\end{align*}
and because $n \geq 0$, the same argument we gave for $K$-theory above
identifies $\partial( GW^{[m]}_n(\mbb G_m/1 \otimes W_\sigma/S^0))$
with $\partial([T]) \cup \partial([T^\sigma]) \cup
GW^{[m]}_{n-1}(W/S^0\otimes W_\sigma/S^0)$ and in turn
with $\mbb GW^{[m]}(X)$. Then we just use the fact that $p^*$ is
injective and a module map to conclude that $\partial([T]) \cup \partial([T^\sigma]) \cup
p^*(GW^{[m]}_{n-1}(W/S^0\otimes W_\sigma/S^0))$ is isomorphic to $\mbb GW^{[m]}(X)$.

Now fix an $i$, and assume by induction that our claim holds for all
$-n \leq i$. Then there's an exact sequence 
\begin{align*}
\mbb GW^{[m]}_n(\mbb A^1/1 \otimes W_\sigma/S^0)  \oplus
  \mbb GW^{[m]}_n(\mbb
  A^1/1 \otimes W_\sigma/S^0) &\rightarrow  \mbb GW^{[m]}_n(\mbb G_m/1
  \otimes W_\sigma/S^0) \\ &\xrightarrow{\partial}  F^{i}_{n-1}(W/S^0 \otimes W_\sigma/S^0)
\end{align*}
which identifies $\partial(\mbb GW^{[m]}_n(\mbb G_m/1
  \otimes W_\sigma/S^0))$ with $\partial([T]) \cup \partial([T^\sigma])
\cup  F^{i}_{n-1}(W/S^0 \otimes W_\sigma/S^0)$, but we know that
  $\partial(\mbb GW^{[m]}_n(\mbb G_m/1
  \otimes W_\sigma/S^0))$ is equal to $\mbb
GW^{[m]}_{n-1}(W/S^0 \otimes W_\sigma/S^0) \cong \mbb
GW^{[m]}_{n-1}(X)$. Thus, letting $p$ denote the projection $W/S^0
\otimes W_\sigma/S^0 \rightarrow X$ to the basepoint, 
\begin{align*}
\mbb GW^{[m]}_{n-1}(X) &\cong p^*(\partial([T]) \cup \partial([T^\sigma])
\cup  F^{i}_{n-1}(W/S^0 \otimes W_\sigma/S^0)) \\ &= \partial([T]) \cup \partial([T^\sigma])
\cup  p^*(F^{i}_{n-1}(W/S^0 \otimes W_\sigma/S^0)) \\  &= F^{i+1}_{n-1}
\end{align*}
since $p^*$ is split injective.

The meatier part of the argument is producing the exact sequence for
$F^{i+1}_{n-1}$, though the proof is essentially the same as the
proof of the base case. 

First note that for all $i$ and $n$, there's a chain complex
\begin{align*}
F^i_n(\mbb A^1/1 \otimes W_\sigma/S^0)  \oplus F^i_n(\mbb
  A^1/1 \otimes W_\sigma/S^0)  &\rightarrow  F^i_n(\mbb G_m/1 \otimes W_\sigma/S^0) \\ 
  &\xrightarrow{\partial} 
  F^{i}_{n-1}(W/S^0 \otimes W_\sigma/S^0) 
\end{align*}
which is just the image of the usual long exact sequence for $GW$
under the map $\gamma^*$. Depending on $n$, this sequence
may or may not be exact, as the image of an exact sequence is in
general not exact. 

Consider the commutative diagram

\[
\resizebox{\textwidth}{!}{
\xymatrix{F^{i+1}_{n-1}(\mbb A^1/1 \otimes W_\sigma/S^0)  \oplus F^{i+1}_{n-1}(\mbb
  A^1/1 \otimes W_\sigma/S^0)  \ar[r] \ar[d]&  \mbb GW(\mbb A^1/1 \otimes W_\sigma/S^0 \otimes (W/S^0 \otimes
  W_\sigma/S^0)^{\otimes i+2})^{\oplus 2} \ar[d] \\
  F^{i+1}_{n-1}(\mbb G_m/1
  \otimes W_\sigma/S^0)  \ar[r]\ar[d]^-{\partial} & \mbb
GW(\mbb G_m/1 \otimes W_\sigma/S^0 \otimes (W/S^0 \otimes W_\sigma/S^0)^{\otimes i+2}) \ar[d]^-{\partial^B}  \\
\ar[r]^{\phi}
  F^{i+1}_{n-2}((W/S^0 \otimes W_\sigma/S^0)^{\otimes i+2})
   &
 \mbb GW^{[m]}_{n-2}((W/S^0 \otimes W_\sigma/S^0)^{\otimes i+3}) }
}
\]
where the upper two horizontal maps are isomorphisms by what we've
already shown. We claim that the left column is exact. The composite is
zero since it's a chain complex, and if $x \in \ker(\partial)$, then
using the fact that the middle and top maps are isomorphisms we
produce a lift of $x$. 

Now it remains only to
check that the image of $\partial$ coincides with
$\partial([T]) \cup \partial([T^\sigma]) \cup
F^{i+1}_{n-2}$. This is the part of the proof we adapt from the
$K$-theory case. First, it's clear that $\partial([T]) \cup \partial([T^\sigma]) \cup
F^{i+1}_{n-2} \subseteq \im(\partial)$, since $\partial([T])$
restricts to zero in $\mbb A^1$. For the other containment, by
exactness and the fact that the left two vertical arrows are
isomorphisms, we know that $\im(\partial) \cong
\im(\partial^B)$. Now since $\partial([T]) \cup \partial([T^\sigma]) \cup
p^*(F^{i+1}_{n-2}) \subseteq \im(\partial)$, it is isomorphic to its
image in $\mbb GW^{[m]}_{n-2}((W/S^0 \otimes W_\sigma/S^0)^{\otimes
  i+2})$. But $\phi$ is a map of modules, so that 
\[
\phi(\partial([T]) \cup \partial([T^\sigma]) \cup
F^{i+1}_{n-2}((W/S^0 \otimes W_\sigma/S^0)^{\otimes i+2})) \cong \partial([T]) \cup \partial([T^\sigma]) \cup \im(\phi)
\] 

But $\phi$ is necessarily surjective, and cup product with
$ \partial([T]) \cup \partial([T^\sigma])$ is an automorphism of $\mbb
GW$. It follows that 
\[
\partial([T]) \cup \partial([T^\sigma]) \cup
F^{i+1}_{n-2}((W/S^0 \otimes W_\sigma/S^0)^{\otimes i+2}) \cong \im
(\phi) = \im(\partial^B) \cong \im(\partial)
\]
so that $\partial([T]) \cup \partial([T^\sigma]) \cup
F^{i+1}_{n-2}((W/S^0 \otimes W_\sigma/S^0)^{\otimes i+2}) =
\im(\partial)$. 

We've shown that if the inductive statement holds for $i,n$,
then it holds for $i+1,n-1$. The fact that it holds for $i+1,m$ for
any $m < n+1$ is
clear by appealing to results for $\mbb GW$. Now, the lemma follows from the explicit description for filtered
colimits of groups.

\end{proof}

\begin{corollary}\label{cor:GW_per}
Let $\gamma$ be the map \eqref{eq:gamma}. Then there are weak
equivalences of presheaves of spectra 
\[
 Q_\gamma GW \xrightarrow{\sim} Q_\gamma \mbb GW \simeq \mbb GW.
\]
\end{corollary}

\begin{proof}
Combining Lemma \ref{lem:pos_htpy} and Lemma \ref{lem:neg_htpy} we see
that $Q_\gamma GW \rightarrow Q_\gamma \mbb GW$ induces an
isomorphism on stable homotopy groups. Lemma \ref{lem:GWb_eq} shows that $Q_\gamma \mbb GW \simeq \mbb GW$.
\end{proof}

Recall the definition of $\beta^{1+\sigma}$ from equation
\eqref{eq:bott_element}.

\begin{definition}
A $GW$-module $E$ is called \emph{Bott periodic} if the map
\[
\iHom(\beta^{1+\sigma},E): E
\rightarrow \iHom((\mbb P^1,[1:1])\wedge (\mbb P^\sigma,[1:1]),E)
\]
is an equivalence.
\end{definition}

Let $\mbb A^-$ and $\mbb G_m^-$ denote $\mbb A^1$ and $\mbb G_m$ with the sign
action $x \mapsto -x$. There are zigzags
\[
\mbb A^1/\mbb G_m \hookrightarrow \mbb P^1/(\mbb P^1 - [-1:1])
\twoheadleftarrow \mbb P^1/[1:1]
\]
and 
\[
\mbb A^-/\mbb G_m^- \hookrightarrow \mbb P^\sigma/(\mbb P^\sigma - [-1:1])
\twoheadleftarrow \mbb P^\sigma/[1:1].
\]

The maps $\beta :  \mbb P^1/[1:1] \rightarrow GW^{[1]}$ and
$\beta^\sigma : \mbb P^\sigma/[1:1] \rightarrow GW^{[-1]}$ lift to
$\mbb P^1/(\mbb P^1 - [-1:1])$ and $\mbb P^\sigma/(\mbb P^\sigma -
[-1:1])$ respectively by results analogous to Lemma
\ref{lem:bott_zero}, and hence there are induced maps 
\[
\beta' : \mbb A^1/\mbb G_m \rightarrow GW^{[1]}
\]
\[
(\beta^\sigma)' : \mbb A^-/\mbb G_m^- \rightarrow GW^{[-1]}.
\]

Taking smash products and using that $\mbb A^1 \oplus \mbb A^- \cong
\mbb A^\rho$, we get a map
\begin{equation}\label{eq:beta_beta_sigma_prime}
\beta' \otimes (\beta^\sigma)' : \mbb
A^\rho/(\mbb A^\rho - 0) \rightarrow GW^{[1]} \otimes GW^{[-1]}
\rightarrow GW.
\end{equation}

When working over a scheme other than the base scheme $S$, we'll let 
$\beta'_X$ and $(\beta^\sigma)'_X$ denote the analogous constructions with 
$\mbb A^1$ and $\mbb G_m$ replaced by $\mbb A^1_X$ and $(\mbb G_m)_X$.
For a vector bundle $E$, let $\mbb V_0(E)$ denote $E/(E-0)$, the
quotient by the complement of the zero section.

\begin{theorem}\label{thm:GW_Einfty}
Let $S$ be a Noetherian scheme of finite Krull dimension with an ample
family of line bundles and $\frac{1}{2} \in S$. Then $L_{\mbb A^1}\mbb GW$ lifts to an $E_\infty$ motivic spectrum, denoted
$\KR_S$, over $\Sm{S,qp}^{C_2}$.
\end{theorem}

\begin{proof}
$GW$ is an $E_\infty$ object in presheaves of spectra (it's a
commutative monoid in the category of presheaves of symmetric spectra) on
$\Sch{S,qp}^{C_2}$ via the cup product defined in \cite[Remark 5.1]{Schder}. By
\cite{cdhdesc} Lemma 3.3, together with corollary \ref{cor:GW_per}
above, $\mbb GW$ is the periodization of $GW$ with respect to
$\gamma$. Let $T^\rho$ denote the Thom space of the regular
representation $\mbb A^\rho$. Now $\mbb GW$ is Nisnevich excisive, so that $\mbb GW(W/S^0
\otimes W_\sigma/S^0) \cong \mbb GW(\mbb P^1 \wedge \mbb P^\sigma)$, and $\mbb GW$ is
$\gamma$ periodic if and only if it is Bott periodic. Because $L_{\mbb A^1}$ preserves 
Nisnevich sheaves of spectra and $E_\infty$-objects, $L_{\mbb A^1}\mbb GW$ is an $E_\infty$ object in the
category of homotopy invariant Nisnevich 
sheaves of spectra. In the notation of Hoyois, let $L_{\mbb A^1}\mbb GW_X$ denote the restriction of 
$L_{\mbb A^1}\mbb GW$ to $X \in \Sch{S,qp}^{C_2}$. Then $L_{\mbb A^1}\mbb GW_X \in \Sp(\HH^{C_2}(X))$, where 
$\Sp(\HH^{C_2}(X))$ is the $\infty$-category of homotopy invariant Nisnevich sheaves of spectra
on $\Sm{X,qp}^{C_2}$. Let $(L_{\mbb A^1}\mbb GW_X)_{mod}$ denote the category of modules over 
$L_{\mbb A^1}\mbb GW_X$ in $\Sp(\HH^{C_2}(X))$. By \cite{cdhdesc}, proposition 3.2, $L_{\mbb A^1}\mbb GW$ lifts to an
$E_\infty$ object $\KR_X$ in $(L_{\mbb A^1}\mbb GW_X)_{mod}[(T^\rho)^{-1}]$ and by 
forgetting the module structure an $E_\infty$-algebra in $\SH^{C_2}(X)$. 
\end{proof}

\begin{lemma}\label{lem:bott_3symm}
The $\mbb A^1$-localization of the Bott element $L_{\mbb A^1}(\beta'_X
\otimes (\beta^\sigma)'_X) : L_{\mbb A^1} \mbb V_0(\mbb A^\rho)
\rightarrow L_{\mbb A^1} GW_X$, viewed as an element of $\Sp(\mc P_{\mbb
  A^1}(\Sch{S,qp}^{C_2}))_{/L_{\mbb A^1}GW_X}$ is 3-symmetric.
\end{lemma}

\begin{proof}
The proof is identical to Lemma 4.8 in \cite{cdhdesc}. The main idea
is that the identity and the cyclic permutation $\sigma_3$ are both
induced by matrices in $SL_{3\cdot 2}(\Z)$ acting on $\mbb A^{3\rho}$,
and any two such matrices are (na\"ively) $\mbb A^1$-homotopic so
that there's a map $h: \mbb A^1 \times \mbb A^{3\rho} \rightarrow \mbb
A^{3\rho}$ witnessing the homotopy. We can extend this to a map
\[
\phi: \mbb A^1 \times \mbb A^{3\rho} \xrightarrow{\pi_1 \times h} \mbb A^1 \times \mbb A^{3\rho}.
\]
Letting $p : \mbb A^1 \times S \rightarrow S$ denote the projection, $\phi$
 is an automorphism of the vector bundle $p^*(\mbb
A^{3\rho})$.

Now we claim that the automorphisms $\phi_0,\phi_1$ of $\mbb
V_0(\mbb A^{3\rho})$ induced by the restrictions of $\phi$ to 1 and 0 are $\mbb
A^1$-homotopic over $L_{\mbb A^1}GW$. Let $\beta'_{\mbb A^{3\rho}}$ denote 
$(\beta'_X \otimes (\beta^\sigma)'_X)^{\otimes 3}$. Let $\beta'_{\mbb A^{3\rho}}$ denote 
$(\beta'_{\mbb A^1 \times X} \otimes (\beta^\sigma)'_{\mbb A^1 \times X})^{\otimes 3}$.

To prove the claim, any automorphism
$\phi$ as above induces a commutative triangle 
\[
\xymatrix{\mbb V_0(p^*(\mbb A^{3\rho})) \ar[rr]^\phi \ar[dr]_{\beta'_{p^*(\mbb A^{3\rho})}} & & \mbb V_0(p^*(\mbb A^{3\rho}))
  \ar[dl]^{\beta'_{p^*(\mbb A^{3\rho})}} \\ & L_{\mbb A^1}GW_{\mbb A^1 \times X}}
\]
or presheaves of spectra on $\Sch{\mbb A^1 \times X}^{C_2}$. As in \cite{cdhdesc}, the diagram 
ultimately comes from our construction of the Bott elements via the projective bundle formula and the functoriality 
of the Proj(Sym - ) construction with respect to automorpshims of the underlying vector bundle 
(in particular the fact that there  
are induced isomorphisms on each twisting sheaf $\mc O(d)$). By
adjunction, this is equivalent to a triangle
\[
\xymatrix{\mbb A^1_+ \otimes \mbb V_0(\mbb A^{3\rho}) \ar[rr] \ar[dr]_{\beta'_{\mbb A^{3\rho}}} && \mbb V_0(\mbb A^{3\rho})
  \ar[dl]^{\beta'_{\mbb A^{3\rho}}} \\ & L_{\mbb A^1}GW_X}
\]
which is an $\mbb A^1$-homotopy between $\phi_0$ and $\phi_1$ over
$L_{\mbb A^1}GW$ as desired.
\end{proof}

We've shown that $\mbb GW$ is Bott periodic and
Nisnevich excisive. Since it's the $\gamma$ periodization of $GW$ and $\gamma$
periodicity is equivalent to Bott periodicity for Nisnevich excisive sheaves, $\mbb GW$ is
in fact the reflection of $GW$
in the subcategory of Nisnevich excisive and Bott periodic
$GW$-modules. Thus by definition, $L_{\mbb A^1}\mbb
GW$ is the reflection of $GW$ in the subcategory of homotopy
invariant, Nisnevich excisive, and Bott periodic $GW$-modules. 

\begin{corollary}\label{cor:MotLoc_GW}
The canonical map $GW \rightarrow Q_{\beta}L_{\mathrm{mot}}GW$ is the universal
map to a homotopy invariant, Nisnevich excisive, and Bott periodic
$GW$-module. In particular
\[
L_{\mbb A^1}\mbb GW = Q_{\beta}L_{\mathrm{mot}}GW
\]
\end{corollary}

\begin{proof}
Given Lemma \ref{lem:bott_3symm}, the proof is identical to
Proposition 4.9 in \cite{cdhdesc}.
\end{proof}

Replacing $GW$ with its connective cover $GW_{\geq 0}$, the same
reasoning yields:
\begin{porism}\label{por:GW_per}
The canonical map $GW_{\geq 0} \rightarrow
Q_{\beta}L_{\mathrm{mot}}GW_{\geq 0}$ is the universal
map to a homotopy invariant, Nisnevich excisive, and Bott periodic
$GW_{\geq 0}$-module. In particular
\[
L_{\mbb A^1}\mbb GW = Q_{\beta}L_{\mathrm{mot}}GW_{\geq 0}
\]
\end{porism}
\begin{proof}
  In short, the reason the result extends to the connective cover $GW_{\geq 0}$ is that we 
  have at no point used the negative homotopy groups of $GW$ in our arguments. We'll spell this out 
  more explicitly now. 

  The connective cover construction is monoidal, and the canonical map $GW^{[m]}_{\geq 0}
  \rightarrow GW^{[m]}$ is a ring map. The Bott elements $\beta$ and $\beta^\sigma$
  live in the zeroth homotopy groups of $GW^{[1]}(\mbb P^1)$ and $GW^{[1]}(\mbb P^\sigma, -\can)$.
  It follows that these Bott elements restrict to well defined elements in the zeroth homotopy groups of 
  $GW^{[1]}_{\geq 0}(\mbb P^1)$ and $GW^{[1]}_{\geq 0}(\mbb P^\sigma, -\can)$. The definition 
  of the map $\gamma$ in \eqref{eq:gamma} extends without modification to $GW_{\geq 0}$, as 
  all the elements involved in the discussion prior to \eqref{eq:gamma} were in the non-negative 
  homotopy groups of $GW$.

  In particular, there's a canonical map 
  $GW^{[m]}_{\geq 0} \rightarrow \mbb GW^{[m]}$ which exhibits $\mbb GW$ as a $GW_{\geq 0}$ module 
  and is an isomorphism on non-negative homotopy 
  groups, and Lemma \ref{lem:pos_htpy} remains true replacing $GW$ with $GW_{\geq 0}$.
  Lemma \ref{lem:neg_htpy}, which is just an analogue of the Bass construction,
   is an inductive argument which at no point uses any facts about the negative homotopy groups 
   of $GW$. The exact sequence involving $GW$ in the proof of Lemma \ref{lem:neg_htpy} is just 
   a formal consequence of the definition of $W$ and $W_\sigma$ (more precisely, that they're 
   pushouts of presheaves of spectra), and remains exact replacing 
   $GW$ with $GW_{\geq 0}$. Finally, the proof of Lemma \ref{lem:bott_3symm} holds without 
   modification when we replace each instance of $GW$ by $GW_{\geq 0}$.
\end{proof}

\section{cdh Descent for Homotopy Hermitian $K$-theory}\label{chap:cdh}
Recall from Definition \ref{def:cdh_top} that the cdh topology is the
topology generated by the Nisnevich and abstract blow-up squares. Fix
a Noetherian scheme of finite Krull dimension $S$, and a scheme $X$
 over $S$.

Let $\HH^{C_2}(S)$ denote the motivic $\infty$-category on
$\Sm{S,qp}^{C_2}$. Just as in \cite{cdhdesc} section 5, we let $\BigHH$ and
$\BigSH$ denote the ``big'' versions of $\HH^{C_2}$ and $\SH^{C_2}$: they can be identified with 
the $\infty$-categories of sections of $\Sp(\HH^{C_2}(-))$ and $\SH^{C_2}(-)$ over $\Sch{S,qp}^{C_2}$
that are cocartesian over smooth morphisms. By the results of the previous section, 
homotopy Hermitian $K$-theory,
$L_{\mbb A^1}\mbb GW$, is a Bott-periodic $E_\infty$-algebra in $\Sp(\BigHH)$, and thus
by \cite[Proposition 3.2]{cdhdesc}, there is a unique Bott periodic $E_\infty$-algebra  
$\BigKR$ in $\BigSH$ such that $\Omega^\infty \BigKR \simeq L_{\mbb A^1}\mbb GW$.
More explicitly, by Porism \ref{por:GW_per}, we can write $\BigKR$ as the image under the localization functor
\[
QL_{\mathrm{mot}} : \mathrm{Stab}^{\mathrm{lax}}_{T^\rho} \Sp(\mc
P(\Sch{S,qp}^{C_2})) \rightarrow \mathrm{Stab}_{T^\rho} \Sp(\BigHH) \simeq \BigSH
\]
of the ``constant'' $T^\rho$-spectrum $c_{\beta' \otimes (\beta^\sigma)'}GW_{\geq 0}$, where the maps
$T^\rho \wedge GW_{\geq 0} \rightarrow GW_{\geq 0}$ are induced by adjunction after
applying $\iHom_{GW_{\geq 0}-mod}(-,GW_{\geq 0})$ to the map 
\[
\beta' \otimes (\beta^\sigma)' : T^\rho \rightarrow GW_{\geq 0}
\]
with $\beta' \otimes (\beta^\sigma)'$ the map \eqref{eq:beta_beta_sigma_prime} restricted 
to the connective cover.
\begin{definition}
For $X \in \Sch{S,qp}^{C_2}$, let $\KR_X$ denote the restriction of $\BigKR$ to $\Sm{X,qp}^{C_2}$. Note that 
this agrees with the notation of Theorem \ref{thm:GW_Einfty}.
\end{definition}

We want to show that $L_{\mbb A^1}\mbb GW$ is a cdh sheaf on $\Sch{S,qp}^{C_2}$. By first
checking that the formalism of six operations holds in equivariant
motivic homotopy theory and following the same recipe as the
$K$-theory case, \cite[Corollary 6.25]{HoyoisSixOp} proves that it suffices to show that for
each $f : D \rightarrow X \in \Sch{S,qp}^{C_2}$, the restriction map
\[
f^*(\KR_X) \rightarrow \KR_D
\]
in $\SH^{C_2}(D)$ is an equivalence. We show this now.

By \cite[Appendix A]{Schder}, there's a map 
\[
\mathrm{Herm}(X)^+ \rightarrow \Omega^\infty GW(X)
\]
where $\mathrm{Herm}(X)$ is the $E_\infty$ space of non-degenerate Hermitian vector
bundles over $X$ and $(-)^+$ denotes group completion. If $X$ is an
affine $C_2$-scheme, the category of vector bundles is a
split exact category with duality, and the above map is an
equivalence. It follows that
\[
\mathrm{Herm}^+ \rightarrow \Omega^\infty GW|\Sm{X,qp}^{C_2}
\]
is a motivic equivalence in $\mc P(\Sm{X,qp}^{C_2})$. 

Just as in \cite{cdhdesc} we note that
\[
\coprod_{n \geq 0} B_{isoEt}O(\langle 1 \rangle^{\perp n}) \rightarrow \mathrm{Herm}
\]
exhibits $\mathrm{Herm}$ as the equivariant Zariski sheafification of
the subgroupoid of non-degenerate Hermitian vector bundles of constant
rank (in other words, it ``corrects" the sections over non-connected or
hyperbolic rings). Since $L_{\mathrm{Zar}}$ preserves finite products,
by \cite[Lemma 5.5]{cdhdesc}, the map remains a Zariski
equivalence after group completion yielding a motivic equivalence
\[
\left(\coprod_{n \geq 0} B_{isoEt}O(\langle 1 \rangle^{\perp
    n})\right)^+ \rightarrow \Omega^\infty GW|\Sm{X,qp}^{C_2}.
\]

Fix a map $f : D \rightarrow X$ in $\Sch{S,qp}^{C_2}$. Again by 
\cite[Lemma 5.5]{cdhdesc}, since the pullback $f^* : \mc P(\Sch{X,qp}^{C_2})
\rightarrow \mc P(\Sch{X,qp}^{C_2})$ preserves finite products, it
commutes with group completion of $E_\infty$-monoids. The same is true
for $L_{\mathrm{mot}}$. It follows that there are motivic equivalences

\[
f^*(\Omega^\infty GW|\Sm{X}^{C_2})\rightarrow  f^*\left(\coprod_{n \geq 0} B_{isoEt}O(\langle 1 \rangle^{\perp
    n})\right)^+ \rightarrow \left(\coprod_{n \geq 0} f^*B_{isoEt}O(\langle 1 \rangle^{\perp
    n})\right)^+
\]

Because $B_{isoEt}O(\langle 1 \rangle^{\perp n})$ is representable by
the results of Section \ref{chap:herm_grass},
\cite [Proposition 2.9]{cdhdesc} yields a motivic equivalence
\[
\left(\coprod_{n \geq 0} f^*B_{isoEt}O(\langle 1 \rangle^{\perp
    n})\right)^+ \rightarrow \left(\coprod_{n \geq 0} B_{isoEt}f^*O(\langle 1 \rangle^{\perp
    n})\right)^+.
\]

But $f^*O(\langle 1 \rangle^{\perp n})|\Sm{X,qp}^{C_2} = O(\langle 1
\rangle^{\perp n})|\Sm{D,qp}^{C_2}$ since $f^*\langle 1 \rangle_X =
\langle 1 \rangle_D$. It follows that there's a motivic equivalence 
\[
 \left(\coprod_{n \geq 0} B_{isoEt}f^*O(\langle 1 \rangle^{\perp
    n})\right)^+ \rightarrow \Omega^\infty GW|\Sm{D,qp}^{C_2},
\]
and combining everything we get that the restriction map
\[
f^*(\Omega^\infty GW|\Sm{X,qp}^{C_2}) \rightarrow \Omega^\infty GW|\Sm{D,qp}^{C_2}
\]
is a motivic equivalence in the $\infty$-category of grouplike
$E_\infty$-monoids in $\mc P(\Sm{D,qp}^{C_2})$. Moving to the category of
connective spectra, It follows that pullback agrees with restriction for $GW_{\geq
  0}$. Because the localization
functor $QL_{\mathrm{mot}}$ is also compatible with the base change
$f^*$, it follows that each arrow 
\[
f^*(QL_{\mathrm{mot}}c_{\beta' \otimes (\beta^\sigma)'}GW_{\geq 0}|\Sm{X,qp}^{C_2}) \rightarrow
QL_{\mathrm{mot}}(f^*c_{\beta' \otimes (\beta^\sigma)'}GW_{\geq 0}|\Sm{X,qp}^{C_2})
\]
\[ 
  \rightarrow QL_{\mathrm{mot}}(c_{\beta' \otimes (\beta^\sigma)'}GW_{\geq 0}|\Sm{D,qp}^{C_2})
\]
is a motivic equivalence. Finally, Porism \ref{por:GW_per} tells us
that $c_{\beta' \otimes (\beta^\sigma)'}GW_{\geq 0}|\Sm{X,qp}^{C_2} \simeq \KR_X$,  so we've proved

\begin{theorem}\label{thm:cdh_desc}
Let $S$ be a Noetherian scheme of finite Krull dimension with an ample
family of line bundles and $\frac{1}{2} \in S$. Then the
homotopy Hermitian $K$-theory spectrum of rings with involution $L_{\mbb A^1}\mbb
GW$ satisfies descent for the equivariant cdh topology on $\Sch{S,qp}^{C_2}$.
\end{theorem}

\bibliographystyle{alpha}

\end{document}